\documentclass[submission]{FPSAC2021}

\usepackage[utf8]{inputenc}
\usepackage[polish, english]{babel}
\usepackage[backend=bibtex, style=alphabetic, doi=false]{biblatex}
\usepackage{tikz}
\usetikzlibrary{arrows.meta}
\usetikzlibrary{calc}
\usepackage{pgfplots}
\usepackage[labelfont={normalsize}]{caption,subfig}
\usepackage{amsmath,amssymb, amsthm}
\usepackage{float}
\usepackage{biblatex}
\usepackage{cleveref}
\usepackage[normalem]{ulem}
\usepackage{comment}

\bibliography{quadratic.bib}

\newtheorem{theorem}{Theorem}
\newtheorem*{conjecture}{Goulden--Rattan conjecture}
\newtheorem{con}{Conjecture}
\newtheorem*{lemma}{Lemma}

\DeclareMathOperator{\degg}{deg}
\DeclareMathOperator{\odd}{odd}
\DeclareMathOperator{\even}{even}
\DeclareMathOperator{\rest}{rest}

\newcommand{\razy}{\circ}
\newcommand{\es}{s}
\newcommand{\ess}{s'}

\newcommand{\nast}
{
	\sigma
}

\newcommand{\lewodd}[4]
{
	\draw[color=#4] [dashed] (#1) to [bend left=#3] (#2);
}

\newcommand{\prawolewodd}[4]
{
	\prawodd{#1}{$.5*(#1)+.5*(#2)$}{#3}{#4}
	\prawodd{#2}{$.5*(#1)+.5*(#2)$}{#3}{#4}
}

\newcommand{\lewoprawoddd}[4]
{
	\prawoddd{#1}{$.5*(#1)+.5*(#2)$}{#3}{#4}
	\prawoddd{#2}{$.5*(#1)+.5*(#2)$}{#3}{#4}
}

\newcommand{\torus}[2]
{

	\centering
	\subfloat[]
	{
		\begin{tikzpicture}[scale=#1]
		\coordinate (x) at (-5.6,-0.7);
		\coordinate (y) at (-2,-2);
		\coordinate (z) at (2,-2);
		\coordinate (t) at (5.6,-0.7);
		\coordinate (p) at (-5,2);
		\coordinate (q) at (0,3);
		\coordinate (r) at (5,2);
		\coordinate (yz) at (1.2,-.85);
		\coordinate (zy) at (-.7,-3.9);
		\coordinate (gora) at (0,.2);
		\coordinate (dol) at (0,-1.1);
		\draw[color=black, ultra thick] (0,0) ellipse (8 and 4);
		\draw[color=black, ultra thick] (5,1) to [bend left=40] (-5,1);
		\draw[color=black, ultra thick] (3.5,0) to [bend right=40] (-3.5,0);
		\prosto{x}{y}{black}
		\prosto{y}{z}{black}
		\prosto{z}{t}{black}
		\lewo{yz}{z}{30}{blue}
		\lewo{zy}{y}{40}{blue}
		\lewoprawoddd{zy}{yz}{40}{blue}
		\lewo{x}{p}{60}{black}
		\lewo{p}{r}{22}{black}
		\lewo{r}{t}{60}{black}
		\wierzcholek{x}{black}{b_2}{dol}
		\wierzcholek{y}{white}{w_1}{gora}
		\wierzcholek{z}{black}{b_1}{dol}
		\wierzcholek{t}{white}{w_2}{dol}
		\end{tikzpicture}
		\label{fig:rystorus}
	}
	\hfill
	\subfloat[]
	{
		\begin{tikzpicture}[scale=#2]
			\punkty		
			\prosto{a}{b}{black}
			\prosto{b}{c}{black}
			\prosto{c}{d}{black}
			\prosto{d}{a}{black}
			\prawolewo{$(c)+(-0.07,0.07)$}{d}{90}{blue}
			\wierzcholek{a}{black}{b_2}{gora}
			\wierzcholek{c}{black}{b_1}{dol}
			\wierzcholek{b}{white}{w_2}{gora}
			\wierzcholek{d}{white}{w_1}{gora}
		\end{tikzpicture}
		\label{fig:rysnormal}
	}	
}

\newcommand{\reform}[1]
{
	\centering
	\subfloat[]
	{
		\begin{tikzpicture}[scale=#1]
			\punkty		
			\prosto{d}{a}{black}
			\lewo{a}{b}{35}{black}
			\lewo{c}{a}{20}{black}
			\prawo{c}{a}{40}{black}
			\prawo{a}{b}{35}{black}
			\wierzcholek{c}{white}{}{dol}
			\wierzcholek{a}{black}{}{gora}
			\wierzcholek{d}{white}{}{gora}
			\wierzcholek{b}{white}{}{gora}
			\draw[red,fill=black] (a) circle (0pt) node[above] at ($(-0.85,0.85)$) {\tiny \small $1$};
			\draw[red,fill=black] (a) circle (0pt) node[above] at ($(-0.55,-0.9)$) {\tiny \small $2$};
			\draw[red,fill=black] (a) circle (0pt) node[above] at ($(1.05,0.5)$) {\tiny \small $3$};
			\draw[red,fill=black] (a) circle (0pt) node[above] at ($(0.45,-0.9)$) {\tiny \small $4$};
			\draw[red,fill=black] (a) circle (0pt) node[above] at ($(0.2,0.3)$) {\tiny \small $5$};
			
			\draw[black,fill=black] (a) circle (0pt) node[above] at ($(3.3,0.8)$) {\tiny \small $\sigma_2=(1\;2\;5\;4\;3)$};
			\draw[black,fill=black] (a) circle (0pt) node[above] at ($(3.5,1.5)$) {\tiny \small $\sigma_1=(1)(2\;4)(3\;5)$};
			\draw[black,fill=black] (a) circle (0pt) node[above] at ($(3.45,0.1)$) {\tiny \small $\sigma_1 \sigma_2=(1\;2\;3\;4\;5)$};
		\end{tikzpicture}
		\label{fig:reformone}
	}
	\hfill
	\subfloat[]
	{
		\begin{tikzpicture}[scale=#1]
			\punkty		
			\prosto{a}{b}{black}
			\prosto{b}{c}{black}
			\prosto{c}{d}{black}
			\prosto{d}{a}{black}
			\prawolewo{$(c)+(-0.07,0.07)$}{d}{90}{black}
			\wierzcholek{a}{black}{}{gora}
			\wierzcholek{c}{black}{}{dol}
			\wierzcholek{b}{white}{}{gora}
			\wierzcholek{d}{white}{}{gora}
			\draw[red,fill=black] (a) circle (0pt) node[above] at ($(-0.95,0.75)$) {\tiny \small $1$};
			\draw[red,fill=black] (a) circle (0pt) node[above] at ($(-0.8,-0.75)$) {\tiny \small $2$};
			\draw[red,fill=black] (a) circle (0pt) node[above] at ($(0.45,0.7)$) {\tiny \small $3$};
			\draw[red,fill=black] (a) circle (0pt) node[above] at ($(0.05,-1.4)$) {\tiny \small $4$};
			\draw[red,fill=black] (a) circle (0pt) node[above] at ($(0.85,-1.4)$) {\tiny \small $5$};
			
			\draw[black,fill=black] (a) circle (0pt) node[above] at ($(-3.5,-1.6)$) {\tiny \small $\sigma_2=(1\;3)(2\;5\;4)$};
			\draw[black,fill=black] (a) circle (0pt) node[above] at ($(-3.5,-0.9)$) {\tiny \small $\sigma_1=(1\;4\;2)(3\;5)$};
			\draw[black,fill=black] (a) circle (0pt) node[above] at ($(-3.45,-2.3)$) {\tiny \small $\sigma_1 \sigma_2=(1\;2\;3\;4\;5)$};
		\end{tikzpicture}
		\label{fig:reformtwo}
	}	
}

\newcommand{\punkty}
{
	\coordinate (a) at (0,2);
	\coordinate (b) at (1.4,0);
	\coordinate (nb) at (1.4,3);
	\coordinate (zab) at (2,0);
	\coordinate (c) at (0,-2);
	\coordinate (d) at (-1.4,0);
	\coordinate (nd) at (-1.4,3);
	\coordinate (s1) at (-2.8,0);
	\coordinate (s2) at (0,4);
	\coordinate (s3) at (0,-4);
	\coordinate (zad) at (-2,0);
	\coordinate (gora) at (0,.3);
	\coordinate (dol) at (0,-.9);
	\coordinate (obok) at (.9,-.9);
	\coordinate (boku) at (.9,0);
	\coordinate (zero) at (0,0);
}

\newcommand{\prosto}[3]
{
	\draw[color=#3] (#1) to [bend left=0] (#2);
}

\newcommand{\lewo}[4]
{
	\draw[color=#4] (#1) to [bend left=#3] (#2);
}

\newcommand{\prawo}[4]
{
	\draw[color=#4] (#1) to [bend right=#3] (#2);
}

\newcommand{\prawolewo}[4]
{
	\prawo{#1}{$.5*(#1)+.5*(#2)$}{#3}{#4}
	\lewo{$.5*(#1)+.5*(#2)$}{#2}{#3}{#4}
}

\newcommand{\lewoprawo}[4]
{
	\lewo{#1}{$.5*(#1)+.5*(#2)$}{#3}{#4}
	\prawo{$.5*(#1)+.5*(#2)$}{#2}{#3}{#4}
}

\newcommand{\waskoszerokodd}[4]
{
	\draw[color=#4] [densely dashed] (#1) to [bend left=0] (#3);
	\draw[color=#4] [densely dashed] (#2) to [bend right=45] (#3);
}

\newcommand{\prawodd}[4]
{
	\draw[color=#4] [dashed] (#1) to [bend right=#3] (#2);
}

\newcommand{\prawoddd}[4]
{
	\draw[color=#4] [dashed] (#1) to [bend right=#3] (#2);
}

\newcommand{\lewoprawodd}[4]
{
	\lewodd{#1}{$.5*(#1)+.5*(#2)$}{#3}{#4}
	\prawodd{$.5*(#1)+.5*(#2)$}{#2}{#3}{#4}
}

\newcommand{\szerokodd}[4]
{
	\draw[color=#4] [dashed] (#1) to [bend left=45] (#3);
	\draw[color=#4] [dashed] (#2) to [bend right=45] (#3);
}

\newcommand{\prostodd}[3]
{
	\draw[color=#3] (#1) [dashed] to [bend left=0] (#2);
}

\newcommand{\wierzcholek}[4]
{
\draw[black,fill=#2] (#1) circle (5pt) node[above] at ($(#1)+(#4)$) {\tiny \large $#3$};
}

\newcommand{\wierzcholekpusty}[4]
{
\draw[black,fill=white] (#1) circle (0pt) node[above] at ($(#1)+(#4)$) {\textcolor{#2}{\tiny \small $#3$}};
}

\newcommand{\ryschange}[1]
{	
	\centering
	\subfloat[]
	{
		\centering
		\begin{tikzpicture}[scale=#1]
			\punkty		
			\prosto{zero}{s1}{black}
			\prawo{zero}{s2}{55}{black}
			\lewo{zero}{s3}{55}{black}
			\prawo{s2}{$.5*(s2)+.5*(s1)$}{55}{black}
			\prawo{s2}{$.65*(s2)+.35*(s1)$}{25}{black}
			\lewo{zero}{$.6*(a)+.3*(d)$}{50}{black}
			\prosto{a}{$.6*(a)+.3*(d)$}{black}
			\prawo{zero}{$.6*(c)+.3*(d)$}{50}{black}
			\prosto{c}{$.6*(c)+.3*(d)$}{black}
			\lewo{s3}{$.5*(s3)+.5*(s1)$}{55}{black}
			\lewo{s3}{$.65*(s3)+.35*(s1)$}{25}{black}
			\wierzcholek{s2}{white}{w_2}{gora}
			\wierzcholek{s2}{white}{\cdots}{dol}
			\wierzcholek{s3}{white}{w_3}{dol}
			\wierzcholek{s3}{white}{\cdots}{gora}
			\wierzcholek{zero}{black}{c_0\Bigg)}{obok}
			\wierzcholek{zero}{black}{\cdots}{gora}
			\wierzcholek{zero}{black}{\cdots}{dol}
			\wierzcholek{s1}{white}{w_1}{gora}
		\end{tikzpicture}
		\label{fig:ryschangea}
	}
	\hspace{70pt}
	\subfloat[]
	{
		\begin{tikzpicture}[scale=#1]
			\punkty		
			\prosto{zero}{s1}{black}
			\lewo{zero}{s3}{55}{black}
			\prosto{zero}{s2}{black}
			\prawo{zero}{$.6*(c)+.3*(d)$}{50}{black}
			\prosto{c}{$.6*(c)+.3*(d)$}{black}
			\lewo{s3}{$.5*(s3)+.5*(s1)$}{55}{black}
			\lewo{s3}{$.65*(s3)+.35*(s1)$}{25}{black}
			\wierzcholek{s2}{white}{w_2}{gora}
			\wierzcholek{s3}{white}{w_3}{dol}
			\wierzcholek{s3}{white}{\cdots}{gora}
			\wierzcholek{zero}{black}{c_0\Bigg)}{$.2*(gora)+.8*(obok)$}
			\wierzcholek{zero}{black}{\cdots}{dol}
			\wierzcholek{s1}{white}{w_1}{gora}
		\end{tikzpicture}
		\label{fig:ryschangeb}
	}
}

\newcommand{\rysall}[1]
{	
	\centering
	\subfloat[]
	{
		\begin{tikzpicture}[scale=#1]
			\punkty		
			\prosto{a}{b}{black}
			\prosto{b}{c}{black}
			\prosto{c}{d}{black}
			\prosto{d}{a}{black}
			\prawolewo{c}{d}{55}{black}
			\wierzcholek{a}{black}{b_2}{gora}
			\wierzcholek{c}{black}{b_1}{dol}
			\wierzcholek{b}{white}{w_2}{gora}
			\wierzcholek{d}{white}{w_1}{gora}
		\end{tikzpicture}
		\label{fig:rysalla}
	}
	\hfill
	\subfloat[]
	{
		\begin{tikzpicture}[scale=#1]
			\punkty		
			\prosto{a}{b}{black}
			\szerokodd{a}{c}{zab}{red}
			\prawolewodd{c}{a}{30}{red}
			\prosto{d}{a}{black}
			\prostodd{c}{a}{red}
			\wierzcholek{a}{black}{b}{gora}
			\wierzcholek{c}{white}{w_3}{dol}
			\wierzcholek{b}{white}{}{gora}
			\wierzcholek{d}{white}{}{gora}
			\kierunek{a}{$(-0.01,-0.39)$}{$(-0.08,-0.15)$}
			\kierunek{a}{$(-0.27,-0.67)$}{$(-0.12,-0.12)$}
			\kierunek{a}{$(0.5,-0.06)$}{$(0.08,-0.15)$}
		\end{tikzpicture}
		\label{fig:rysallb}
	}
	\hfill
	\subfloat[]
	{
		\begin{tikzpicture}[scale=#1]
			\punkty		
			\prosto{d}{a}{black}
			\lewo{a}{b}{35}{black}
			\lewodd{c}{a}{20}{red}
			\prawodd{c}{a}{40}{red}
			\prawo{a}{b}{35}{black}
			\wierzcholek{c}{white}{w_3}{dol}
			\wierzcholek{a}{black}{b}{gora}
			\wierzcholek{d}{white}{}{gora}
			\wierzcholek{b}{white}{}{gora}
			\kierunek{a}{$(-0.22,-0.67)$}{$(-0.15,-0.08)$}
			\kierunek{a}{$(+0.3,-0.43)$}{$(-0.08,-0.15)$}
		\end{tikzpicture}
		\label{fig:rysallc}
	}
	\hfill
	\subfloat[]
	{
		\begin{tikzpicture}[scale=#1]
			\punkty		
			\prosto{a}{b}{black}
			\waskoszerokodd{a}{c}{$.6*(b)+.4*(c)$}{blue}
			\prostodd{c}{a}{red}
			\prosto{d}{a}{black}
			\prawolewodd{c}{a}{30}{red}
			\wierzcholek{a}{black}{b}{gora}
			\wierzcholek{c}{white}{w_3}{dol}
			\wierzcholek{b}{white}{}{gora}
			\wierzcholek{d}{white}{}{gora}
			\kierunek{a}{$(-0.01,-0.39)$}{$(-0.08,-0.15)$}
			\kierunek{a}{$(-0.27,-0.67)$}{$(-0.12,-0.12)$}
			\kierunek{a}{$(0.18,-0.58)$}{$(0.15,-0.08)$}
		\end{tikzpicture}
		\label{fig:rysalld}
	}
}
\newcommand{\rysodd}[1]
{	
	\centering
	\subfloat[]
	{
		\begin{tikzpicture}[scale=#1]
			\punkty		
			\prosto{a}{b}{black}
			\prostodd{b}{c}{red}
			\prostodd{d}{c}{red}
			\prosto{d}{a}{black}
			\prawolewodd{c}{d}{55}{red}
			\wierzcholek{a}{black}{b_2}{gora}
			\wierzcholek{c}{black}{b_1}{dol}
			\wierzcholek{b}{white}{w_2}{gora}
			\wierzcholek{d}{white}{w_1}{gora}
			\kierunek{b}{$(-0.28,-0.41)$}{$(0.15, -0.08)$}
			\kierunek{d}{$(0.29,-0.4)$}{$(0.12,0.12)$}
			\kierunek{d}{$(0.02,-0.49)$}{$(0.15,0.08)$}
		\end{tikzpicture}
		\label{fig:rysodda}
	}
	\hfill
	\subfloat[]
	{
		\begin{tikzpicture}[scale=#1]
			\punkty		
			\prosto{a}{b}{black}
			\szerokodd{$.5*(a)+.5*(b)$}{c}{$(zab)+(0.0,0.25)$}{red}
			\prawolewodd{c}{$.5*(a)+.5*(d)$}{40}{red}
			\prosto{d}{a}{black}
			\prostodd{c}{$.5*(a)+.5*(d)$}{red}
			\wierzcholek{a}{black}{b_2}{gora}
			\wierzcholek{c}{black}{b_1}{dol}
			\wierzcholek{b}{white}{w_2}{gora}
			\wierzcholek{d}{white}{w_1}{gora}
			\kierunek{$.5*(a)+.5*(b)$}{$(0.28,0.04)$}{$(-0.08,0.15)$}
			\kierunek{$.5*(a)+.5*(d)$}{$(0.09,-0.36)$}{$(0.12,0.12)$}
			\kierunek{$.5*(a)+.5*(d)$}{$(-0.13,-0.49)$}{$(0.15,0.08)$}
		\end{tikzpicture}
		\label{fig:rysoddb}
	}
	\hfill
	\subfloat[]
	{
		\begin{tikzpicture}[scale=#1]
			\punkty		
			\prosto{a}{b}{black}
			\szerokodd{a}{c}{zab}{red}
			\prawolewodd{c}{a}{30}{red}
			\prosto{d}{a}{black}
			\prostodd{c}{a}{red}
			\wierzcholek{a}{black}{b_2}{gora}
			\wierzcholek{c}{black}{b_1}{dol}
			\wierzcholek{b}{white}{w_2}{gora}
			\wierzcholek{d}{white}{w_1}{gora}
			\kierunek{a}{$(-0.01,-0.39)$}{$(-0.08,-0.15)$}
			\kierunek{a}{$(-0.27,-0.67)$}{$(-0.12,-0.12)$}
			\kierunek{a}{$(0.5,-0.06)$}{$(0.08,-0.15)$}
		\end{tikzpicture}
		\label{fig:rysoddc}
	}
	\hfill
	\subfloat[]
	{
		\begin{tikzpicture}[scale=#1]
			\punkty		
			\prosto{a}{b}{black}
			\szerokodd{a}{c}{zab}{red}
			\prawolewodd{c}{a}{30}{red}
			\prosto{d}{a}{black}
			\prostodd{c}{a}{red}
			\wierzcholek{a}{black}{b}{gora}
			\wierzcholek{c}{white}{w_3}{dol}
			\wierzcholek{b}{white}{w_2}{gora}
			\wierzcholek{d}{white}{w_1}{gora}
			\kierunek{a}{$(-0.01,-0.39)$}{$(-0.08,-0.15)$}
			\kierunek{a}{$(-0.27,-0.67)$}{$(-0.12,-0.12)$}
			\kierunek{a}{$(0.5,-0.06)$}{$(0.08,-0.15)$}
		\end{tikzpicture}
		\label{fig:rysoddd}
	}
}

\newcommand{\ryseven}[1]
{	
	\centering
	\subfloat[]
	{
		\begin{tikzpicture}[scale=#1]
			\punkty		
			\prostodd{a}{b}{red}
			\prosto{b}{c}{black}
			\lewo{c}{d}{30}{black}
			\prawodd{a}{$0.5*(a)+0.2*(c)$}{15}{red}
			\szerokodd{$0.5*(a)+0.2*(c)$}{$(d)+(0.08,0.00)$}{$(d)+(1.0,-0.8)$}{red}
			\prawo{c}{d}{30}{black}
			\wierzcholek{a}{black}{b_2}{gora}
			\wierzcholek{c}{black}{b_1}{dol}
			\wierzcholek{b}{white}{w_2}{gora}
			\wierzcholek{d}{white}{w_1}{gora}
			\kierunek{d}{$(0.25,-0.45)$}{$(0.17,0)$}
			\kierunek{b}{$(-0.23,0.32)$}{$(-0.12,-0.12)$}
		\end{tikzpicture}
		\label{fig:rysevena}
	}
	\hfill
	\subfloat[]
	{
		\begin{tikzpicture}[scale=#1]
			\punkty		
			\prostodd{a}{$.5*(b)+.5*(c)$}{red}
			\prosto{b}{c}{black}
			\lewo{c}{d}{30}{black}
			\szerokodd{$.5*(c)+.5*(d)+(0.4,-0.0)$}{a}{$(d)+(0.65,-0.85)$}{red}
			\prawo{c}{d}{30}{black}
			\wierzcholek{a}{black}{b_2}{gora}
			\wierzcholek{c}{black}{b_1}{dol}
			\wierzcholek{b}{white}{w_2}{gora}
			\wierzcholek{d}{white}{w_1}{gora}
			\kierunek{$0.5*(c)+0.2*(d)$}{$(-0.29,-0.01)$}{$(0,-0.17)$}
			\kierunek{$0.5*(b)+0.2*(c)$}{$(-0.08,-0.24)$}{$(-0.12,-0.12)$}
		\end{tikzpicture}
		\label{fig:rysevenb}
	}
	\hfill
	\subfloat[]
	{
		\begin{tikzpicture}[scale=#1]
			\punkty		
			\prosto{b}{c}{black}
			\lewo{c}{d}{35}{black}
			\lewodd{a}{c}{20}{red}
			\prawodd{a}{c}{40}{red}
			\prawo{c}{d}{35}{black}
			\wierzcholek{a}{black}{b_2}{gora}
			\wierzcholek{c}{black}{b_1}{dol}
			\wierzcholek{b}{white}{w_2}{gora}
			\wierzcholek{d}{white}{w_1}{gora}
			\kierunek{c}{$(0.22,0.67)$}{$(0.15,0.08)$}
			\kierunek{c}{$(-0.3,0.43)$}{$(0.08,0.15)$}
		\end{tikzpicture}
		\label{fig:rysevenc}
	}
	\hfill
	\subfloat[]
	{
		\begin{tikzpicture}[scale=#1]
			\punkty		
			\prosto{b}{c}{black}
			\lewo{c}{d}{35}{black}
			\lewodd{a}{c}{20}{red}
			\prawodd{a}{c}{40}{red}
			\prawo{c}{d}{35}{black}
			\wierzcholek{a}{white}{w_3}{gora}
			\wierzcholek{c}{black}{b}{dol}
			\wierzcholek{b}{white}{w_2}{gora}
			\wierzcholek{d}{white}{w_1}{gora}
			\kierunek{c}{$(0.22,0.67)$}{$(0.15,0.08)$}
			\kierunek{c}{$(-0.3,0.43)$}{$(0.08,0.15)$}
		\end{tikzpicture}
		\label{fig:rysevend}
	}
}

\newcommand{\rysrest}[1]
{	
	\centering
	\subfloat[]
	{
		\begin{tikzpicture}[scale=#1]
			\punkty		
			\prosto{a}{b}{black}
			\prostodd{b}{c}{blue}
			\prostodd{d}{c}{red}
			\prosto{d}{a}{black}
			\prawolewodd{c}{d}{55}{red}
			\wierzcholek{a}{black}{b_2}{gora}
			\wierzcholek{c}{black}{b_1}{dol}
			\wierzcholek{b}{white}{w_2}{gora}
			\wierzcholek{d}{white}{w_1}{gora}
			\kierunek{b}{$(-0.29,-0.42)$}{$(-0.08, 0.15)$}
			\kierunek{d}{$(0.29,-0.4)$}{$(0.12,0.12)$}
			\kierunek{d}{$(0.02,-0.49)$}{$(0.15,0.08)$}
		\end{tikzpicture}
		\label{fig:rysresta}
	}
	\hfill
	\subfloat[]
	{
		\begin{tikzpicture}[scale=#1]
			\punkty		
			\prosto{a}{b}{black}
			\waskoszerokodd{$.5*(a)+.5*(b)$}{c}{$.6*(b)+.4*(c)$}{blue}
			\prostodd{c}{$.5*(a)+.5*(d)$}{red}
			\prosto{d}{a}{black}
			\prawolewodd{c}{$.5*(a)+.5*(d)$}{40}{red}
			\wierzcholek{a}{black}{b_2}{gora}
			\wierzcholek{c}{black}{b_1}{dol}
			\wierzcholek{b}{white}{w_2}{gora}
			\wierzcholek{d}{white}{w_1}{gora}\kierunek{$.5*(a)+.5*(d)$}{$(0.09,-0.36)$}{$(0.12,0.12)$}
			\kierunek{$.5*(a)+.5*(d)$}{$(-0.13,-0.49)$}{$(0.15,0.08)$}
			\kierunek{$.5*(a)+.5*(b)$}{$(0.02,-0.42)$}{$(-0.08, 0.15)$}
		\end{tikzpicture}
		\label{fig:rysrestb}
	}
	\hfill
	\subfloat[]
	{
		\begin{tikzpicture}[scale=#1]
			\punkty		
			\prosto{a}{b}{black}
			\waskoszerokodd{a}{c}{$.6*(b)+.4*(c)$}{blue}
			\prostodd{c}{a}{red}
			\prosto{d}{a}{black}
			\prawolewodd{c}{a}{30}{red}
			\wierzcholek{a}{black}{b_2}{gora}
			\wierzcholek{c}{black}{b_1}{dol}
			\wierzcholek{b}{white}{w_2}{gora}
			\wierzcholek{d}{white}{w_1}{gora}
			\kierunek{a}{$(-0.01,-0.39)$}{$(-0.08,-0.15)$}
			\kierunek{a}{$(-0.27,-0.67)$}{$(-0.12,-0.12)$}
			\kierunek{a}{$(0.18,-0.58)$}{$(0.15,-0.08)$}
		\end{tikzpicture}
		\label{fig:rysrestc}
	}
	\hfill
	\subfloat[]
	{
		\begin{tikzpicture}[scale=#1]
			\punkty		
			\prosto{a}{b}{black}
			\waskoszerokodd{a}{c}{$.6*(b)+.4*(c)$}{blue}
			\prostodd{c}{a}{red}
			\prosto{d}{a}{black}
			\prawolewodd{c}{a}{30}{red}
			\wierzcholek{a}{black}{b}{gora}
			\wierzcholek{c}{white}{w_3}{dol}
			\wierzcholek{b}{white}{w_2}{gora}
			\wierzcholek{d}{white}{w_1}{gora}
			\kierunek{a}{$(-0.01,-0.39)$}{$(-0.08,-0.15)$}
			\kierunek{a}{$(-0.27,-0.67)$}{$(-0.12,-0.12)$}
			\kierunek{a}{$(0.18,-0.58)$}{$(0.15,-0.08)$}
		\end{tikzpicture}
		\label{fig:rysrestd}
	}
}

\newcommand{\kierunek}[3]
{
	 \draw [-{Triangle[angle=60:0.7mm]}] ($(#1)+(#2)$) -- ($(#1)+(#2)+(#3)$);
}

\newcommand{\sliding}[1]
{
	\centering
	\subfloat[]
	{
		\begin{tikzpicture}[scale=#1]
			\punkty		
			\prostodd{d}{a}{red}
			\lewoprawodd{d}{a}{50}{red}
			\lewodd{a}{d}{75}{red}
			\lewoprawo{b}{d}{35}{black}
			\prawo{d}{$(0,0)$}{35}{black}
			\prosto{$(0,0)$}{$(1,0.6)$}{black}
			\lewo{$(1,0.6)$}{b}{90}{black}
			\lewo{b}{d}{65}{black}
			\wierzcholek{a}{black}{}{gora}
			\wierzcholek{b}{black}{}{gora}
			\wierzcholek{d}{black}{}{$(dol)+(-0.2,0.0)$}
			\kierunek{d}{$(0.04,0.48)$}{$(0.17,0)$}
			\kierunek{d}{$(0.29,0.41)$}{$(0.17,0)$}
			\kierunek{d}{$(0.5,-0)$}{$(0.15,-0.08)$}
		\end{tikzpicture}
		\label{fig:rysprzed}
	}
\hfill
	\subfloat[]
	{
		\begin{tikzpicture}[scale=#1]
			\punkty		
			\prostodd{$(d)+(0.7,0.23)$}{a}{red}
			\lewoprawodd{$(d)+(0.7,0.23)$}{a}{40}{red}
			\lewodd{a}{$(d)+(0.7,-0.23)$}{45}{red}
			\lewoprawo{b}{d}{35}{black}
			\prawo{d}{$(0,0)$}{35}{black}
			\prosto{$(0,0)$}{$(1,0.6)$}{black}
			\lewo{$(1,0.6)$}{b}{90}{black}
			\lewo{b}{d}{65}{black}
			\wierzcholek{a}{black}{}{gora}
			\wierzcholek{b}{black}{}{gora}
			\wierzcholek{d}{black}{}{$(dol)+(-0.2,0.0)$}
			\kierunek{$(d)+(0.7,0.23)$}{$(-0.04,0.28)$}{$(0.12,0.12)$}
			\kierunek{$(d)+(0.7,0.23)$}{$(0.11,0.26)$}{$(0.17,0)$}
			\kierunek{$(d)+(0.7,-0.23)$}{$(0.27,0.17)$}{$(0.17,0)$}
		\end{tikzpicture}
		\label{fig:rysw}
	}
\hfill
	\subfloat[]
	{
		\begin{tikzpicture}[scale=#1]
			\punkty		
			\prostodd{b}{a}{red}
			\lewoprawodd{b}{a}{50}{red}
			\lewodd{a}{b}{100}{red}
			\lewoprawo{b}{d}{35}{black}
			\prawo{d}{$(0,0)$}{35}{black}
			\prosto{$(0,0)$}{$(1,0.6)$}{black}
			\lewo{$(1,0.6)$}{b}{90}{black}
			\lewo{b}{d}{65}{black}
			\wierzcholek{a}{black}{}{gora}
			\wierzcholek{b}{black}{}{gora}
			\wierzcholek{d}{black}{}{$(dol)+(-0.2,0.0)$}
			\wierzcholek{d}{black}{}{$(dol)+(-0.2,0.0)$}
			\kierunek{b}{$(-0.45,0.2)$}{$(-0.12,-0.12)$}
			\kierunek{b}{$(-0.29,0.41)$}{$(-0.12,-0.12)$}
			\kierunek{b}{$(0.24,0.34)$}{$(-0.12,0.12)$}
		\end{tikzpicture}
		\label{fig:ryspo}
	}
}

\newcommand{\slidone}[1]
{
	\centering
	\subfloat[]
	{
		\begin{tikzpicture}[scale=#1]
			\punkty		
			
			\prostodd{d}{a}{red}
			\wierzcholekpusty{$.5*(a)+.5*(d)$}{black}{4}{$.3*(gora)$}
			
			\wierzcholekpusty{$.5*(b)+.5*(nb)$}{black}{1}{$.8*(boku)$}
			\wierzcholekpusty{$.5*(d)+.5*(nd)$}{black}{2}{$-.8*(boku)$}
			\wierzcholekpusty{$.5*(d)+.5*(b)$}{black}{3}{$1.05*(dol)$}
			
			\lewo{b}{d}{30}{black}
			\lewo{d}{nd}{30}{black}
			\lewo{nb}{b}{30}{black}
			\wierzcholek{a}{black}{}{gora}
			\wierzcholek{b}{black}{c'}{$(dol)+(0.2,-0.05)$}
			\wierzcholek{nb}{black}{}{gora}
			\wierzcholek{nd}{black}{}{gora}
			\wierzcholek{d}{black}{c}{$(dol)+(-0.2,-0.05)$}
			\kierunek{d}{$(0.29,0.41)$}{$(0.12,-0.12)$}
		\end{tikzpicture}
		\label{fig:onebefore}
	}
	\hfill
	\subfloat[]
	{
			\begin{tikzpicture}[scale=#1]
			\punkty		
			
			\prostodd{$.5*(b)+.5*(d)-(0,.4)$}{a}{red}
			\wierzcholekpusty{$.4*(a)+.27*(d)+.13*(b)$}{black}{4}{$-0.8*(gora)$}
			
			\wierzcholekpusty{$.5*(b)+.5*(nb)$}{black}{1}{$.8*(boku)$}
			\wierzcholekpusty{$.5*(d)+.5*(nd)$}{black}{2}{$-.8*(boku)$}
			\wierzcholekpusty{$.5*(d)+.5*(b)$}{black}{3}{$1.05*(dol)$}
			
			\lewo{b}{d}{30}{black}
			\lewo{d}{nd}{30}{black}
			\lewo{nb}{b}{30}{black}
			\wierzcholek{a}{black}{}{gora}
			\wierzcholek{b}{black}{c'}{$(dol)+(0.2,-0.05)$}
			\wierzcholek{nb}{black}{}{gora}
			\wierzcholek{nd}{black}{}{gora}
			\wierzcholek{d}{black}{\es(c)}{$(dol)+(-0.2,-0.2)$}
			\kierunek{$.5*(d)+.5*(b)-(0,.4)$}{$(0.01,0.49)$}{$(0.17,0)$}
		\end{tikzpicture}
		\label{fig:oneduring}
	}
	\hfill
	\subfloat[]
	{
			\begin{tikzpicture}[scale=#1]
			\punkty			
			
			\prostodd{b}{a}{red}
			\wierzcholekpusty{$.4*(a)+.4*(b)+.1*(d)$}{black}{4}{$-0.8*(gora)$}
			
			\wierzcholekpusty{$.5*(b)+.5*(nb)$}{black}{1}{$.8*(boku)$}
			\wierzcholekpusty{$.5*(d)+.5*(nd)$}{black}{2}{$-.8*(boku)$}
			\wierzcholekpusty{$.5*(d)+.5*(b)$}{black}{3}{$1.05*(dol)$}
			
			\lewo{b}{d}{30}{black}
			\lewo{d}{nd}{30}{black}
			\lewo{nb}{b}{30}{black}
			\wierzcholek{a}{black}{}{gora}
			\wierzcholek{b}{black}{\ess(c')}{$(dol)+(0.2,-0.2)$}
			\wierzcholek{nb}{black}{}{gora}
			\wierzcholek{nd}{black}{}{gora}
			\wierzcholek{d}{black}{\es(c)}{$(dol)+(-0.2,-0.2)$}
			\kierunek{b}{$(-0.29,0.41)$}{$(-0.12,-0.12)$}
		\end{tikzpicture}
		\label{fig:oneafter}
	}
}

\newcommand{\slidthree}[1]
{
	\centering
	\subfloat[]
	{
		\begin{tikzpicture}[scale=#1]
			\punkty		
			
			\prostodd{d}{a}{red}
			\wierzcholekpusty{$.6*(a)+.4*(d)$}{black}{5}{$.3*(gora)$}
			
			\lewoprawodd{d}{a}{50}{red}
			\wierzcholekpusty{$.1*(a)+.9*(d)$}{black}{6}{$1.7*(gora)$}
			
			\lewodd{$.4*(a)+.3*(b)+.3*(d)$}{d}{10}{red}
			\lewodd{a}{$.4*(a)+.3*(b)+.3*(d)$}{45}{red}
			\wierzcholekpusty{$.1*(a)+.6*(b)+.3*(d)$}{black}{4}{$1.7*(gora)$}
			
			\wierzcholekpusty{$.5*(b)+.5*(nb)$}{black}{1}{$.8*(boku)$}
			\wierzcholekpusty{$.5*(d)+.5*(nd)$}{black}{2}{$-.8*(boku)$}
			\wierzcholekpusty{$.5*(d)+.5*(b)$}{black}{3}{$1.05*(dol)$}
			
			\lewo{b}{d}{30}{black}
			\lewo{d}{nd}{30}{black}
			\lewo{nb}{b}{30}{black}
			\wierzcholek{a}{black}{}{gora}
			\wierzcholek{b}{black}{c'}{$(dol)+(0.2,-0.05)$}
			\wierzcholek{nb}{black}{}{gora}
			\wierzcholek{nd}{black}{}{gora}
			\wierzcholek{d}{black}{c}{$(dol)+(-0.2,-0.05)$}
			\kierunek{d}{$(0.04,0.48)$}{$(0.12,-0.12)$}
			\kierunek{d}{$(0.29,0.41)$}{$(0.12,-0.12)$}
			\kierunek{d}{$(0.42,0.16)$}{$(0.12,-0.12)$}
		\end{tikzpicture}
		\label{fig:threebefore}
	}
	\hfill
	\subfloat[]
	{
		\begin{tikzpicture}[scale=#1]
			\punkty		
			
			\prostodd{$.5*(b)+.5*(d)-(0,.4)$}{a}{red}
			\wierzcholekpusty{$.6*(a)+.27*(d)+.13*(b)$}{black}{5}{$-0.8*(gora)$}
			
			\lewoprawodd{$.5*(b)+.5*(d)-(0,.4)$}{a}{75}{red}
			\wierzcholekpusty{$.1*(a)+.25*(b)+.65*(d)$}{black}{6}{$-0.5*(gora)$}
			
			\lewodd{$.6*(a)+.8*(b)+.2*(d)$}{$.5*(b)+.5*(d)-(0,.4)$}{15}{red}
			\lewodd{a}{$.6*(a)+.8*(b)+.2*(d)$}{60}{red}
			\wierzcholekpusty{$.1*(a)+.8*(b)+.1*(d)$}{black}{4}{$1.5*(gora)$}
			
			\wierzcholekpusty{$.5*(b)+.5*(nb)$}{black}{1}{$.8*(boku)$}
			\wierzcholekpusty{$.5*(d)+.5*(nd)$}{black}{2}{$-.8*(boku)$}
			\wierzcholekpusty{$.5*(d)+.5*(b)$}{black}{3}{$1.05*(dol)$}
			
			\lewo{b}{d}{30}{black}
			\lewo{d}{nd}{30}{black}
			\lewo{nb}{b}{30}{black}
			\wierzcholek{a}{black}{}{gora}
			\wierzcholek{b}{black}{c'}{$(dol)+(0.2,-0.05)$}
			\wierzcholek{nb}{black}{}{gora}
			\wierzcholek{nd}{black}{}{gora}
			\wierzcholek{d}{black}{\es_3(c)}{$(dol)+(-0.2,-0.2)$}
			\kierunek{$.5*(d)+.5*(b)-(0,.4)$}{$(-0.31,0.37)$}{$(0.17,0)$}
			\kierunek{$.5*(d)+.5*(b)-(0,.4)$}{$(0.01,0.49)$}{$(0.17,0)$}
			\kierunek{$.5*(d)+.5*(b)-(0,.4)$}{$(0.3,0.34)$}{$(0.17,0)$}
		\end{tikzpicture}
		\label{fig:threeduring}
	}
	\hfill
	\subfloat[]
	{
		\begin{tikzpicture}[scale=#1]
			\punkty			
			
			\prostodd{b}{a}{red}
			\wierzcholekpusty{$.6*(a)+.27*(b)+.13*(d)$}{black}{5}{$-0.8*(gora)$}
			
			\lewoprawodd{b}{a}{50}{red}
			\wierzcholekpusty{$.1*(a)+.25*(d)+.65*(b)$}{black}{6}{$-0.5*(gora)$}
			
			\lewodd{$.45*(nb)+.45*(b)+.1*(a)$}{b}{10}{red}
			\lewodd{a}{$.45*(nb)+.45*(b)+.1*(a)$}{60}{red}
			\wierzcholekpusty{$.1*(a)+.8*(b)+.1*(d)$}{black}{4}{$6.0*(gora)$}
			
			\wierzcholekpusty{$.5*(b)+.5*(nb)$}{black}{1}{$.8*(boku)$}
			\wierzcholekpusty{$.5*(d)+.5*(nd)$}{black}{2}{$-.8*(boku)$}
			\wierzcholekpusty{$.5*(d)+.5*(b)$}{black}{3}{$1.05*(dol)$}
			
			\lewo{b}{d}{30}{black}
			\lewo{d}{nd}{30}{black}
			\lewo{nb}{b}{30}{black}
			\wierzcholek{a}{black}{}{gora}
			\wierzcholek{b}{black}{\ess_3(c')}{$(dol)+(0.2,-0.2)$}
			\wierzcholek{nb}{black}{}{gora}
			\wierzcholek{nd}{black}{}{gora}
			\wierzcholek{d}{black}{\es_3(c)}{$(dol)+(-0.2,-0.2)$}
			\wierzcholek{d}{black}{}{$(dol)+(-0.2,0.0)$}
			\kierunek{b}{$(-0.45,0.2)$}{$(-0.12,-0.12)$}
			\kierunek{b}{$(-0.29,0.41)$}{$(-0.12,-0.12)$}
			\kierunek{b}{$(0.01,0.62)$}{$(-0.12,-0.12)$}
		\end{tikzpicture}
		\label{fig:threeafter}
	}
}

\newcommand{\multicycles}[1]
{
	\subfloat[]
	{
		\begin{tikzpicture}[scale=#1]
			\punkty		
			
			\prostodd{d}{a}{black}
			\wierzcholekpusty{$.6*(a)+.4*(d)$}{black}{5}{$.3*(gora)$}
			
			\lewoprawodd{d}{a}{50}{black}
			\wierzcholekpusty{$.1*(a)+.9*(d)$}{black}{6}{$1.7*(gora)$}
			
			\lewodd{$.4*(a)+.3*(b)+.3*(d)$}{d}{10}{black}
			\lewodd{a}{$.4*(a)+.3*(b)+.3*(d)$}{45}{black}
			\wierzcholekpusty{$.1*(a)+.6*(b)+.3*(d)$}{black}{4}{$1.7*(gora)$}
			
			\wierzcholekpusty{$.5*(b)+.5*(nb)$}{black}{1}{$.8*(boku)$}
			\wierzcholekpusty{$.5*(d)+.5*(nd)$}{black}{2}{$-.8*(boku)$}
			\wierzcholekpusty{$.5*(d)+.5*(b)$}{black}{3}{$1.05*(dol)$}
			
			\lewo{b}{d}{30}{black}
			\lewo{d}{nd}{30}{black}
			\lewo{nb}{b}{30}{black}
			\wierzcholek{a}{black}{}{gora}
			\wierzcholek{b}{black}{}{$(dol)+(0.2,-0.05)$}
			\wierzcholek{nb}{black}{}{gora}
			\wierzcholek{nd}{black}{}{gora}
			\wierzcholek{d}{black}{}{$(dol)+(-0.2,-0.05)$}
		\end{tikzpicture}
		\label{fig:multicyclesone}
	}
	\hfill
	\subfloat[]
	{
		\begin{tikzpicture}[scale=#1]
			\punkty			
			
			\prostodd{b}{a}{black}
			\wierzcholekpusty{$.6*(a)+.27*(b)+.13*(d)$}{black}{5}{$-0.8*(gora)$}
			
			\lewoprawodd{b}{a}{50}{black}
			\wierzcholekpusty{$.1*(a)+.25*(d)+.65*(b)$}{black}{6}{$-0.5*(gora)$}
			
			\lewodd{$.45*(nb)+.45*(b)+.1*(a)$}{b}{10}{black}
			\lewodd{a}{$.45*(nb)+.45*(b)+.1*(a)$}{60}{black}
			\wierzcholekpusty{$.1*(a)+.8*(b)+.1*(d)$}{black}{4}{$6.0*(gora)$}
			
			\wierzcholekpusty{$.5*(b)+.5*(nb)$}{black}{1}{$.8*(boku)$}
			\wierzcholekpusty{$.5*(d)+.5*(nd)$}{black}{2}{$-.8*(boku)$}
			\wierzcholekpusty{$.5*(d)+.5*(b)$}{black}{3}{$1.05*(dol)$}
			
			\lewo{b}{d}{30}{black}
			\lewo{d}{nd}{30}{black}
			\lewo{nb}{b}{30}{black}
			\wierzcholek{a}{black}{}{gora}
			\wierzcholek{b}{black}{}{$(dol)+(0.2,-0.2)$}
			\wierzcholek{nb}{black}{}{gora}
			\wierzcholek{nd}{black}{}{gora}
			\wierzcholek{d}{black}{}{$(dol)+(-0.2,-0.2)$}
			\wierzcholek{d}{black}{}{$(dol)+(-0.2,0.0)$}
		\end{tikzpicture}
		\label{fig:multicyclestwo}
	}
}

\title{Quadratic coefficients of Goulden--Rattan character polynomials}

\author[Mikołaj Marciniak]{Mikołaj Marciniak\thanks{\href{mailto:marciniak@mat.umk.pl}{marciniak@mat.umk.pl}. Mikołaj Marciniak was supported by Narodowe Centrum Nauki, grant number 2017/26/A/ST1/00189 and Narodowe Centrum Badań i Rozwoju, grant number POWR.03.05.00-00-Z302/17-00.}\addressmark{1}}

\address{\addressmark{1}Interdisciplinary Doctoral School “Academia Copernicana”, Faculty of Mathematics and Computer Science, Nicolaus Copernicus University in Toruń, ul.~Chopina 12/18, 87-100 Toruń, Poland}

\received{\today}

\abstract{ Goulden--Rattan polynomials give 
the exact value of the subdominant part of 
the normalized characters of the symmetric 
groups in terms of certain quantities $(C_i)$ 
which describe the macroscopic shape of the 
Young diagram. 
The Goulden--Rattan positivity 
conjecture states that the coefficients of 
these polynomials are positive rational 
numbers with small  denominators. We prove 
a special case of this conjecture for the 
coefficient of the quadratic term $C_2^2$ by 
applying certain bijections involving maps 
(i.e., graphs drawn on surfaces).}

\keywords{characters of the symmetric groups, 
free cumulants, Kerov polynomials, 
Goulden--Rattan polynomials, maps}

\begin{document}

\maketitle

%\newpage

\section{Introduction}

\subsection{Normalized characters}

Characters are a basic tool of representation
theory. After normalization, they are also
useful in asymptotic problems.

If $k\leq n$ are natural numbers, then any
permutation $\pi \in S_k$ can also be treated 
as an element of the larger symmetric group 
$S_n$ by adding $n-k$ additional fixpoints. 
For any permutation $\pi \in S_k$ and any
irreducible representation $\rho^{\lambda}$ 
of the symmetric group $S_n$ which corresponds 
to the Young diagram $\lambda$, we define 
\emph{the normalized character}
$$
\Sigma_{\pi}(\lambda)
= \begin{cases} 
          n(n-1)\cdots(n-k+1)\frac{\operatorname{Tr}\rho^{\lambda}(\pi)}{\text{dimension of }\rho^{\lambda}} & \text{for } k \leq n,\\
          0 & \text{otherwise.}
\end{cases}
$$

Of particular interest are the character
values on the cycles, therefore we
will use the shorthand notation
$$\Sigma_{k}(\lambda)=\Sigma_{(1,2,\ldots,k)}(\lambda).$$

\subsection{Free cumulants}

\emph{Free cumulants} are an important tool
of free probability theory \cite{VDN92} and
random matrix theory \cite{Voi91}. In the
context of the representation theory of the
symmetric groups they can be defined as follows,
see \cite{Bia03}. For a Young diagram $\lambda$
we define its free cumulants 
$R_2(\lambda), R_3(\lambda), \ldots$ as
$$R_k(\lambda)=\lim_{s\to\infty} \frac{1}{s^k}\Sigma_{k-1}(s\lambda),$$
where the diagram $s\lambda$ is created from 
the diagram $\lambda$ by dividing each box 
of $\lambda$ into an $s\times s$ square. 

The free cumulants have been defined in such a way 
as to be very helpful for studying asymptotic behaviour 
of the characters on a cycle of length $k$ when the size 
of the Young diagram tends to infinity \cite{Bia98}.

\subsection{Kerov character polynomials}

Kerov \cite{Ker00} formulated the following result: 
for each permutation $\pi$ and any Young
diagram $\lambda$, the normalized 
character $\Sigma_{\pi}(\lambda)$ is 
equal to the value of some polynomial 
$K_{\pi}(R_2(\lambda), R_3(\lambda), \ldots)$
(now called \emph{the Kerov character polynomial})
with integer coefficients. The first 
published proof of this fact was provided 
by Biane \cite{Bia03}. The Kerov character
polynomial is \emph{universal} because it 
does not depend on the choice of $\lambda$. 
We are interested in the values of the 
characters on cycles, therefore for 
$\pi=(1, 2, \ldots, k)$ we use 
the simplified notation
\begin{align}
\label{kerpol}
\Sigma_k=K_k(R_2, R_3, \ldots)
\end{align}
for such Kerov polynomials. The first few Kerov polynomials $K_k$ are as follows:
\begin{align*}
K_1&=R_2,\\
K_2&=R_3,\\
K_3&=R_4+R_2,\\
K_4&=R_5+5R_3,\\
K_5&=R_6+15R_4+5R_2^2+8R_2,\\
K_6&=R_7+35R_5+35R_3R_2+84R_3, \\
K_7&=R_8+180R_2+224R_2^2+14R_2^3+56R_3^2+469R_4+84R_2R_4+70R_6.
\end{align*}

Kerov conjectured that the coefficients of 
the polynomial $K_k$ are \emph{non-negative} integers.
Goulden and Rattan \cite{GR05} found an explicit 
formula for the coefficients of the Kerov
polynomial $K_k$; unfortunately, their formula 
was complicated and did not give any combinatorial
interpretation to the coefficients. Later, 
F\'{e}ray proved positivity \cite{Fer09} and together 
with Dołęga and Śniady found a combinatorial 
interpretation of the coefficients \cite{DFS10}.
In this paper, we will use the combinatorial
interpretation given by them in the special case 
of linear and square coefficients.

\subsection{Goulden--Rattan conjecture}

Goulden and Rattan \cite{GR05} introduced 
a family of functions $C_2,C_3,\dots$ on 
the set of Young diagrams given by
$C_0=1$, $C_1=0$ and
$$C_k^{\lambda}=\frac{24}{k(k+1)(k+2)}\lim_{s\to\infty}\frac{1}{s^k}\Big(\Sigma_{k+1}(s\lambda)-R_{k+2}(s\lambda)\Big)$$
for $k\geq2$.

Śniady \cite{Sni06} proved the explicit form of $C_k$
(conjectured by Biane \cite{Bia03}) as a polynomial in the 
free cumulants $R_2, R_3, \ldots$ given by
\begin{align}
\label{cformula}
C_k=\sum_{\substack{j_2,j_3,\ldots \geq 0 \\ 2j_2+3j_3+\cdots=k}}
(j_2+j_3+\cdots)!\prod_{i\geq 2} \frac{\big( (i-1)R_i\big)^{j_i}}{j_i!}
\end{align}
for $k\geq 2$.
The aforementioned formula of Goulden and Rattan for 
the Kerov polynomials was naturally expressed in 
terms of these quantities $C_2,C_3,\dots$ \cite{GR05}. 
More specifically, they constructed an explicit 
polynomial $L_k$ with rational coefficients such that 
\begin{align}
\label{grpol}
K_k-R_{k+1}=L_k(C_2, C_3, \ldots).
\end{align}
These polynomials are called 
\emph{the Goulden--Rattan polynomials}. 
They formulated the following conjecture:
\begin{conjecture}
\label{hipotezaGR}
The coefficients of the Goulden--Rattan polynomials 
are non-negative numbers with small denominators. 
\end{conjecture}
The first few Goulden--Rattan polynomials are 
as follows \cite{GR05}:
\begin{align*}
K_1&-R_2=0\\
K_2&-R_3=0\\
K_3&-R_4=C_2,\\
K_4&-R_5=\frac{5}{2}C_3,\\
K_5&-R_6=5C_4+8C_2,\\
K_6&-R_7=\frac{35}{4}C_5+42C_3,\\
K_7&-R_8=14C_6+\frac{469}{3}C_4+\frac{203}{3}C_2^2+180C_2.
\end{align*}

Linear coefficients of the Goulden--Rattan polynomials 
are non-negative, because they are equal to certain 
scaled coefficients of the Kerov polynomial:
\[ [ C_j ] L_k=\frac{1}{j-1} [ R_j ] K_k. \]

In this paper we will prove that the coefficient 
of $C_2^2$ is non-negative. We hope that edge sliding we will define in this article
will also be a useful tool in proving non-negativity of 
the square coefficients $[ C_i C_j ] L_k$.
The next step towards the proof of 
Goulden--Rattan conjecture would be to understand 
the cubic coefficients $ [ C_i C_j C_u ] L_k$; we 
hope that our methods will still be applicable there,
nevertheless, there seem to be some difficulties 
related to the inclusion-exclusion principle.

\subsection{Graphs on surfaces, maps and expanders}

We will consider graphs drawn on an oriented
surface. Each face of such a graph has some 
number of edges ordered cyclically by going 
along the boundary of the face  and touching 
it with the right hand. We will call it 
\emph{the clockwise boundary direction}. If we use the left hand 
and visit the edges in the opposite order, we 
will call it {the counterclockwise boundary direction}. 

By \emph{a map} we mean a bipartite 
graph drawn without intersections on an oriented 
and connected surface with minimal genus. The maps 
which we consider have a fixed choice of colouring 
of the vertices, i.e., each 
vertex is coloured black or white, with the edges
connecting the vertices of the opposite colours. 
An example of a map is shown in \cref{fig:mapa}.

\begin{figure}[H]
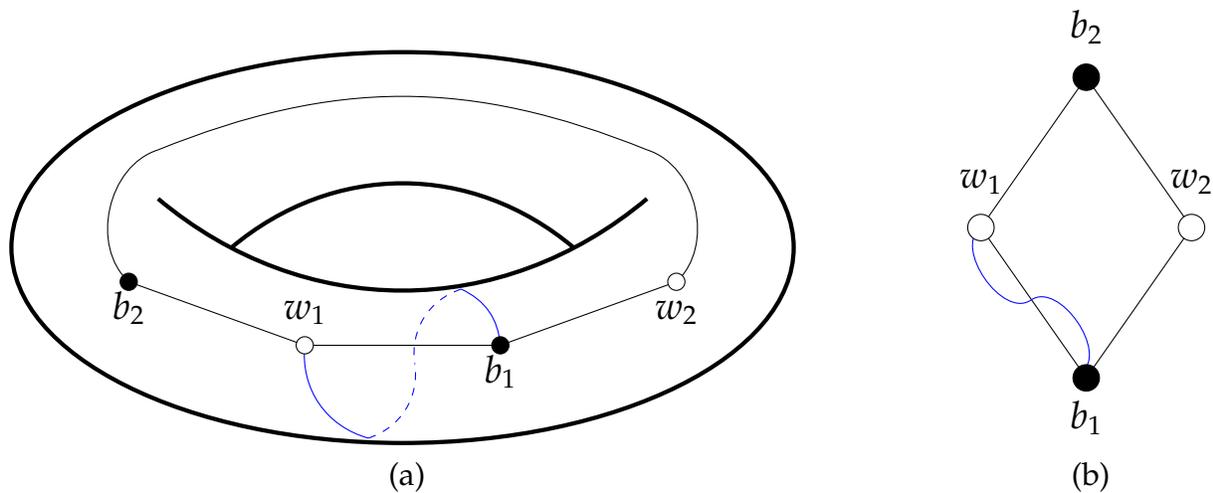

\torus{0.65}{1.0}
\caption
{
	\protect\subref{fig:rystorus} 
		An example of a map with $4$ vertices 
		and $5$ edges drawn on a torus.\\
	\protect\subref{fig:rysnormal} 
		The same map drawn for simplicity 
		on the plane. 
}
\label{fig:mapa}
\end{figure}

\emph{An expander} \cite[Appendix A.1]{Sni19} 
 is a map with the following properties.
\begin{itemize}

\item It has a distinguished edge (known as
the root) and one face.

\item Each black vertex is assigned a natural 
number, known as a weight, such that each 
non-empty proper subset of the set of black 
vertices has more white neighbours than the sum 
of its weights.

\item The sum of all weights is equal to 
the number of white vertices.

\end{itemize}
The map from \cref{fig:mapa} is an expander 
if each black vertex has weight $1$ 
(any choice of the root is valid).

Using the Euler characteristic we get
\begin{align}
\label{euler}
2-2g=\chi=V-k+1
\end{align}
where $g$ denotes the genus of the surface, 
$V$ denotes the number of vertices and $k$ 
denotes the number of edges.

\subsection{Combinatorial interpretation of 
the Kerov polynomial coefficients}
The following two theorems 
\cref{thm1} and \cref{thm3}
% given by Biane and Féray
give a combinatorial interpretation to 
the linear and square coefficients of 
the Kerov character polynomials 
\cite[Theorem 1.2, Theorem 1.3]{DFS10}.
The first is as follows.
\begin{theorem}
\label{thm1} 
For all integers $l\geq2$ and $k\geq 1$ 
the coefficient $[R_l] K_k$ 
is equal to the number of pairs 
$(\sigma_1,\sigma_2)$ of permutations 
$\sigma_1, \sigma_2 \in S(k)$ such that
$\sigma_1\sigma_2 =(1 \; 2 \; \cdots \;k)$ and 
such that $\sigma_2$ consists of one
cycle and $\sigma_1$ consists of $l-1$ cycles.
(We use the convention $ \sigma_1 \sigma_2 = \sigma_2 ( \sigma_1 ) =  \sigma_2 \circ \sigma_1$. )
\end{theorem}
The expanders are a graphical interpretation 
of these pairs of permutations.
There is a natural bijection between such 
pairs of permutations $(\sigma_1,
\sigma_2)$ and the expanders with one face, 
one black vertex, $l-1$ white
vertices and $k$ edges. Additionally, one edge 
is selected as the root and the unique black 
vertex has a weight $l-1$. More precisely:

\begin{itemize}

\item The edges are numbered $1, 2, \dots, k$.
The edge with number $1$ is selected as the root.

\item \emph{The counterclockwise angular cyclic order of the 
edges on a given vertex} (i.e., the order of edges 
ending at this vertex around it) corresponds to 
a cycle of a permutation
depending on the colour of this vertex, i.e., 
$\sigma_1$ for white and $\sigma_2$ for black
(in this case we have a unique cycle of 
the permutation $\sigma_2$). 

\item The unique face corresponds to the unique 
counterclockwise boundary
cycle of the permutation $(1 \; 2 \; \cdots \; k)$.

\end{itemize}
Since there is only one face, the root determines the 
numbering of all edges. We can reformulate \cref{thm1} as follows. 
(See \cref{fig:reformone} for an example.)
\begin{theorem}
\label{thm2}
For all integers $l\geq2$ and $k\geq 1$ the 
coefficient $[R_l] K_k$ is equal to the number 
of expanders with $k$ edges, $l-1$ white vertices 
and $1$ black vertex with the weight $l-1$.
\end{theorem}

\begin{figure}[H]
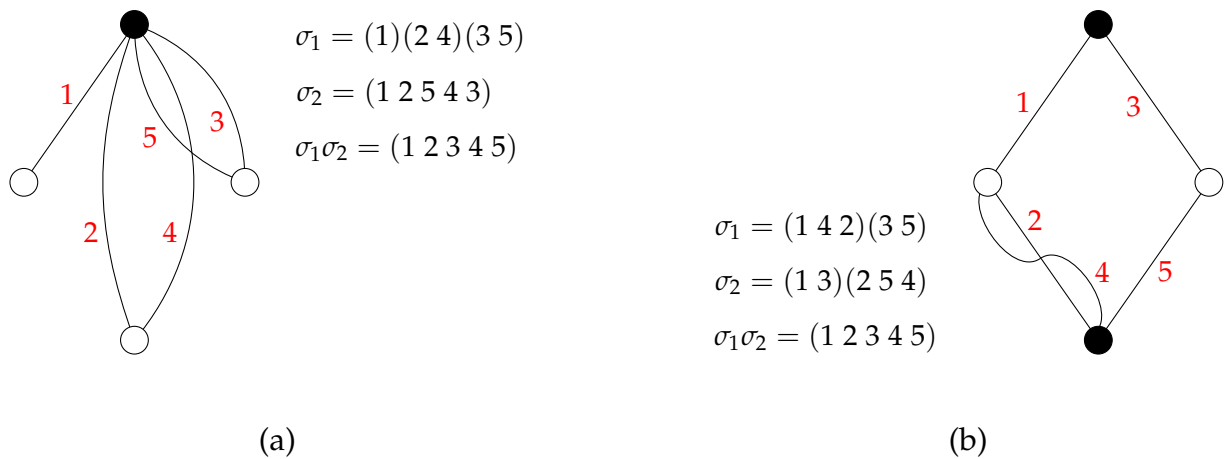

\reform{1.05}
\caption
{
		Examples of expanders with $5$ edges and 
		their corresponding pair of permutations 
		$\sigma_1, \sigma_2$ such that 
		$\sigma_1 \sigma_2 = (1 \; 2\; 3 \; 4 \; 5)$. 
		The root is assigned the number $1$. 
		\protect\subref{fig:reformone} 
		The expander with one black vertex 
		and three white vertices. 
		\protect\subref{fig:reformtwo}
		The expander with two black vertices 
		and two white vertices.
}
\end{figure}

Similarly we use the second theorem \cite[Theorem 1.3]{DFS10} 
for square coefficients.
\begin{theorem}
\label{thm3}
For all integers $l_1, l_2\geq 2$ and $k\geq 1$ 
the coefficient $[R_{l_1} R_{l_2}] K_k$ is equal 
to the number of triples $(\sigma_1, \sigma_2, q)$ 
with the following properties.
\begin{itemize}

\item The permutations $\sigma_1, \sigma_2 \in S_k$ 
fulfill the equality $\sigma_1 \sigma_2=(1 \; 2 \; \cdots \; k)$.

\item The permutation $\sigma_1$ consists of two cycles and 
the permutation $\sigma_2$ consists of $l_1+l_2-2$ cycles.

\item The function $q$ associates the numbers 
$l_1$ and $l_2$ to the two cycles of $\sigma_1$. 
Furthermore, for each cycle $c$ of $\sigma_1$ there 
exist at least $q(c)$ cycles of $\sigma_2$ which nontrivially
intersect $c$.
\end{itemize}
\end{theorem}
Analogously, we can also reformulate \cref{thm3}. 
(See \cref{fig:reformtwo} for an example.)
\begin{theorem}
\label{thm4}
For all integers $l_1, l_2\geq 2$ and $k\geq 1$ 
the coefficient $[R_{l_1} R_{l_2}] K_k$ is equal 
to the number of expanders with $k$ edges, 
$l_1+l_2-2$ white vertices and $2$ black 
vertices with weights $l_1-1, l_2-1$.
\end{theorem}

\subsection{Relationship between coefficients 
of Goulden--Rattan polynomials 
and coefficients of Kerov polynomials}

The formula \eqref{cformula} allows us to 
express $(C_i)$ in terms of free cumulants;
we see that the coefficients of the terms 
$R_i R_j$, $R_{i+j}$, $R_j^2$ and $R_{2j}$ in the expressions
$C_{i} C_{j}$, $C_{i+j}$, $C^2_{j}$ and $C_{2j}$ are given for $i \neq j$ by
\begin{align*}
C_i C_j &= (i-1)(j-1) R_i R_j + 0  R_{i+j} + \text{(sum of other terms)},\\
C_{i+j} &=2(i-1)(j-1)R_iR_j+(i+j-1)R_{i+j}+\text{(sum of other terms)},\\
C_{j}^{2} &=(j-1)^2R_j^2+0 R_{2j}+\text{(sum of other terms)},\\
C_{2j} &=(j-1)^2R_j^2+(2j-1)R_{2j}+\text{(sum of other terms)}. 
\end{align*}
Moreover, any product $C_{i_1} C_{i_2}\cdots C_{i_t}$
of at least $t\geq 3$ factors does not contain any 
of the terms $C_i C_j$, $C_{i+j}, C^2_{j}$ and $C_{2j}$.
It follows that the square coefficients of 
the Goulden--Rattan polynomial are related 
to the coefficients of the Kerov polynomial via
\begin{align*}
\frac{\partial^2L_k}{\partial{C_i}\partial{C_j}}\Bigg|_{0=C_1=C_2=\cdots}
&=\frac{1}{(i-1)(j-1)}\frac{\partial^2K_k}{\partial{R_i}\partial{R_j}}\Bigg|_{0=R_1=R_2=\cdots}
-2\frac{\partial L_k}{\partial{C_{i+j}}}\Bigg|_{0=C_1=C_2=\cdots}\\
&=\frac{1}{(i-1)(j-1)}\frac{\partial^2K_k}{\partial{R_i}\partial{R_j}}\Bigg|_{0=R_1=R_2=\cdots}
-\frac{2}{(i+j-1)}\frac{\partial K_k}{\partial{R_{i+j}}}\Bigg|_{0=R_1=R_2=\cdots}
\end{align*}
for $i \neq j$. Whereas the quadratic coefficients are related via
\begin{align*}
\frac{\partial^2L_k}{\partial{C_j^2}}\Bigg|_{0=C_1=C_2=\cdots}
&=\frac{1}{(j-1)^2}\frac{\partial^2K_k}{\partial{R_j^2}}\Bigg|_{0=R_1=R_2=\cdots}
-2 \frac{\partial L_k}{\partial{C_{2j}}}\Bigg|_{0=C_1=C_2=\cdots}\\
&=\frac{1}{(j-1)^2}\frac{\partial K_k^2}{\partial{R_j^2}}\Bigg|_{0=R_1=R_2=\cdots}
-\frac{2}{(2j-1)}\frac{\partial K_k}{\partial{R_{2j}}}\Bigg|_{0=R_1=R_2=\cdots}
\end{align*}
for any natural number $j$.
Thus, we obtain the explicit formula for 
the square coefficients of the Goulden–Rattan polynomial:
\begin{align}
\label{formula}
[C_j^2]L_k &=\frac{1}{(j-1)^2}[R_j^2]K_k-\frac{1}{2j-1}[R_{2j}]K_k \\
\intertext{and}
[C_iC_j]L_k &=\frac{1}{(i-1)(j-1)}[R_iR_j]K_k-\frac{2}{i+j-1}[R_{i+j}]K_k
\qquad \text{for $i\neq j$.}
\end{align}

\section{The main result} 

Let $Y_k(u)$ denote the set of expanders 
with $k$ edges, $u-1$ white vertices and 
one black vertex. Let $X_k(i, j)$ denote 
the set of expanders with $k$ edges, 
$i+j-2$ white vertices and two black 
vertices with weights $i-1$ and $j-1$. 
Using  \cref{thm2} and \cref{thm4} we can 
also reformulate the Goulden--Rattan 
conjecture for the square coefficients 
in terms of expanders, as follows.
\begin{con} 
    \label{con:GJ2}
Let $i \neq j$ be natural numbers. Then 
\begin{align*}
(2j-1)\ \left\| X_k(j, j) \right\|    &   \geq (j-1)^2\ \left\| Y_k(2j) \right\| \\
\intertext{and}
(i+j-1)\ \left\| X_k(i, j) \right\|  & \geq 2(i-1)(j-1)\ \left\| Y_k(i+j) \right\|
\end{align*}
for any natural number $k$. 
\end{con}
These inequalities are equivalent to the 
positivity of the coefficients $[C_j^2] L_k$ 
and $[C_i C_j] L_k$ respectively. 
In this text we prove only the first 
inequality  in the special case 
$j=2$. We hope to present
a proof of \cref{con:GJ2} in its general 
form in a future paper.

Using \cref{formula} 
we can calculate several examples of 
the coefficient of $C_2^2$ of 
the Goulden--Rattan polynomials
\begin{align*}
[C_2^2] L_4 &=0-0=0,\\
[C_2^2] L_5 &=5-\frac{1}{3} \cdot 15=0,\\
[C_2^2] L_6 &=0-0=0, \\
[C_2^2] L_7 &=224-\frac{1}{3} \cdot 469=\frac{203}{3}, \\
[C_2^2] L_8 &=0-0=0. 
\end{align*}
Note that if $k$ is even then $[C_2^2] L_k=0$ because there does not exist an expander with $4$ vertices and an even number of edges, since
$2-2g=2j-k+1$ by \cref{euler}. Additionally, $[C_2^2]L_1=0$ and $[C_2^2]L_3=0$. Thus we can assume that the number of edges $k$ is odd and $k \geq 5$. 

Let 
\begin{align}
\label{xdef}
X_k &= X_k(2, 2), \\
\label{ydef}
Y_k &= Y_k(4).
\end{align}
The set $X_k$ consists of expanders with 
$2$ black vertices and $2$ white vertices
such that each black vertex is connected 
with both white vertices;
each black vertex necessarily has weight 
equal to $1$. The set $Y_k$ consists of 
expanders with one black vertex 
(which necessarily has weight $3$) 
connected with all $3$ white vertices. 
From now on we will omit the weights 
of the black vertices.

The main goal of this paper is to prove 
the following:
\begin{theorem}
\label{mainthm2}
The inquality
$$3\big\|X_k\big\| \geq \big\|Y_k\big\|$$
is true for any natural number $k$.
\end{theorem}

\section{Maps, expanders, and edge sliding}

In this section, we provide some necessary 
background details for the proof of \cref{mainthm2}.
We will denote a transposition exchanging $a$ and $h$ by $(a \; h)$.
For any set $H \subset \{1, 2, \dots\}$ such that $a, h \in H$, 
we can treat the transposition $(a \; h)$ as a permutation of the set $H$.
(By adding fixed points to the transposition.) Therefore, 
for any permutation $\pi$ of the set $H$, 
the products $\pi \razy (a \; h)$ and $(a \; h) \razy \pi$ 
are also permutations of the set $H$. 

Let $G$ be a graph with edges numbered $1, \dots, k$ drawn without 
intersections on an oriented and connected surface with minimal genus. 
There is a natural bijection between the graph $G$ and a multiset of cycles $M_G$
(cyclic permutations of a subset of $\{1, \dots, k\}$) such that
each number $j$ belongs to exactly two cycles. 
More precisely, the counterclockwise angular cyclic order of 
edges at a given vertex of the graph $G$ (i.e., the order
of edges ending at this vertex around it) corresponds to a single cycle of $M_G$.  
Two examples of a graph with its multiset of cycles are shown in \cref{fig:multi-cycles}. 

\begin{figure}[H]
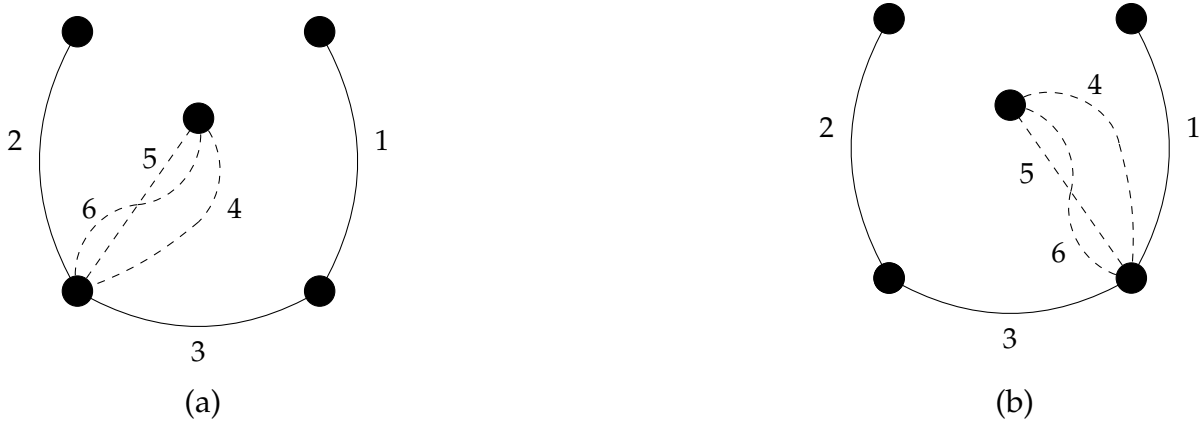

	\multicycles{1.15}
	\caption
	{
		\protect\subref{fig:multicyclesone} The graph with the multiset of cycles $\big\{ (2) (1) (5 \; 6 \; 4) (2 \; 3 \; 4 \; 5 \; 6)(3 \; 1)\big\}$. \linebreak[4]{}
		\protect\subref{fig:multicyclestwo} The graph with the multiset of cycles $\big\{ (2) (1) (5 \; 6 \; 4) (2 \; 3 ) ( 4 \; 5 \; 6 \; 3 \; 1)\big\}$.
	}
	\label{fig:multi-cycles}
\end{figure}

In this section, each end of every edge (or equivalently, each element of every cycle) will have one of three values assigned to it: clockwise direction, counterclockwise direction, or no direction. Note that the two ends of the same edge can be assigned different values. By default, we assume that a cycle element to which we have not assigned a direction has an assigned the value no direction. 

\subsection{Edge sliding for a single number}
Let $M_G$ be the multiset of cycles of a graph $G$. Let $a$ belonging to a cycle $c \in M_G$ be a number with assigned 
the clockwise direction. 
(The number "a" from a cycle different from "c" and any number different from "a" 
from any cycle do not have a direction assigned to them.)
Let $c' \neq c$ be the second cycle of the multiset $M_G$ containing the number $h=c^{-1}(a)$.
We assume that the number $a$ does not belong to the cycle $c'$. 
We define \emph{the single edge sliding for the number $a$ in the clockwise boundary direction from the cycle $c$ to the cycle $c'$ along the number $h$}
as the replacement of the cycles $c, c'$ in the multiset $M_G$ by two new cycles $\es(c), \ess(c')$ given by 
\begin{align*}
\es(c)&= \left[ c \razy \left(a \; h\right) \right] \setminus \left(a\right),\\ 
\ess(c')&= (a \; h) \razy \left[ \left( a \right) \; c' \right] .
\end{align*}
Note that the product  $c \razy \left(a \; h\right)$
consists of two cycles: the cycle $(a)$ and the cycle 
formed from the cycle $c$ by removing the number $a$.
Therefore, the above transformation removes the number $a$ from the 
cycle $c$ and adds it to the cycle $c'$ before the number $h$.
Finally, we change the direction of the number $a$ to the counterclockwise direction.  
%We will say that the number $a$ is sliding along the number $h$. 
We can naturally think of such operations as sliding of edges in a graph. 
An example of the single edge sliding 
is shown in \cref{fig:sliding-one-edge}.

\begin{figure}
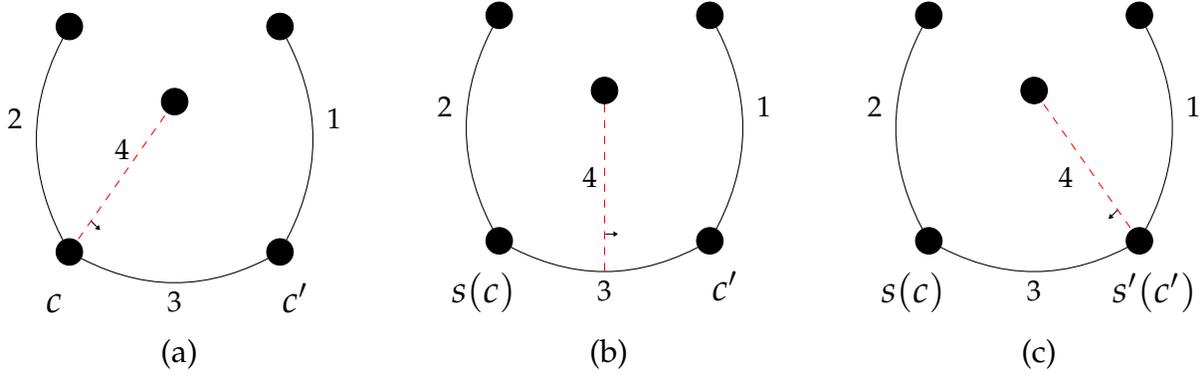

	\slidone{1}
	\caption
	{
		An example of the single edge sliding for the number ${a=4}$ in the clockwise boundary direction 
		from the cycle ${c=(2 \; 3 \; 4)}$ to the cycle ${c'=(1 \; 3)}$ along the number ${h=3}$.
		\protect\subref{fig:onebefore} The initial graph $G$. The edge labelled $a$ is dashed and coloured red.
		\protect\subref{fig:oneduring} Visualization of the single edge sliding.
		\protect\subref{fig:oneafter} The resulting graph with two updated cycles: 
		${\es(c)= \left[ (2 \; 3 \; 4) \razy \left(3 \; 4\right) \right] \setminus \left(4\right) 
		= \left(2 \; 3 \right)}$ and 
		${\ess(c')= (3 \; 4) \razy \left[ \left( 1 \; 3\right) \; \left( 4 \right)\right] = (1 \; 4 \; 3).}$
	}
	\label{fig:sliding-one-edge}
\end{figure}

Analogously, we define the single edge sliding for the number $a$ in \emph{the counterclockwise boundary direction} 
from the cycle $c$ to the cycle $c'$ along \emph{the number $h=c(a)$} as 
the replacement of the cycles $c, c'$ in the multiset $M_G$ by two new cycles $\es(c), \ess(c')$ given by 
\begin{align*}
	\es(c)&= \left[ \left(a \; h\right) \razy c \right] \setminus \left(a\right),\\ 
	\ess(c')&= \left[ \left( a \right) \; c'\right] \razy (a \; h).
\end{align*}

\subsection{Edge sliding for a sequence of numbers}
Let $M_G$ be the multiset of cycles of a graph $G$. 
Let $a_1, \dots, a_l$ be a sequence of successive but not all numbers belonging to a cycle $c\in M_G$, i.e. such that $a_j=c^{j-1}(a_1)$ for each index $j \leq l$. Moreover, to each of the numbers $a_1, \dots, a_l$ is assigned the clockwise direction. Let $c' \neq c$ be the second cycle of the multiset of $M_G$ containing the number $h=c^{-1}(a_1)$.
We assume that the numbers $a_1, \dots, a_l$ do not belong to the cycle $c'$. 
We define \emph{the package edge sliding for the sequence of the numbers $a_1, \dots, a_l$ in the clockwise boundary direction from the cycle $c$ to the cycle $c'$ along the number $h$}
as the replacement of the cycles $c, c'$ in the multiset $M_G$ by two new cycles $\es_l(c), \ess_l(c')$ obtained by recursion
with initial conditions $\es_0(c)=c, \ess_0(c')=c'$ and a recursive step given for each $1 < j \leq l$ by
\begin{align*}
	\es_{j}(c)&= \left[ \es_{j-1}(c) \razy \left(a_j \; h\right) \right] \setminus \left(a_j \right),\\ 
	\ess_{j}(c')&= (a_j \; h) \razy \left[ \left( a_j \right) \; \ess_{j-1}(c') \right] .
\end{align*}
Note that the package edge sliding for a sequence of numbers $a_1, \dots, a_l$ is actually equivalent to the sequential single edge sliding for numbers $a_1, \dots, a_l$.
Therefore, the above transformation removes the numbers $a_1, \dots a_l$ from the 
cycle $c$ and adds them in the same order to the cycle $c'$ before the number $h$.
Finally, we change the directions of the numbers $a_1, \dots, a_l$ to the counterclockwise boundary direction.   
%We will say that each of the numbers $a_1, \dots, a_l$ is sliding along the number $h$. 
Note that the single edge sliding is a special case of 
the package edge sliding.
An example of the package edge sliding is shown in \cref{fig:sliding-three-edge}.

\begin{figure}[H]
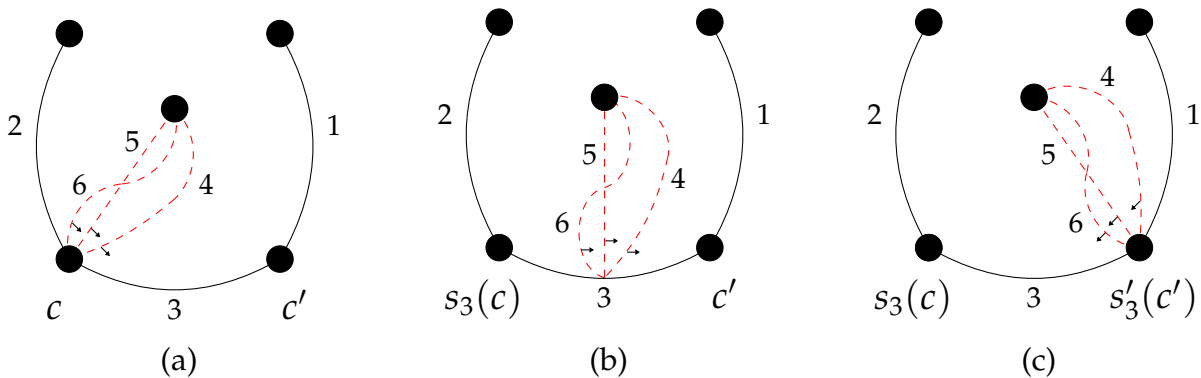

	\slidthree{1}
	\caption
	{
		An example of the package edge sliding for the numbers ${a_1=4}$, ${a_2=5}$, ${a_3=6}$ in the clockwise boundary direction 
		from the cycle ${c=(2 \; 3 \; 4 \; 5 \; 6)}$ to the cycle ${c'=(1 \; 3)}$ along the number ${h=3}$.
		\protect\subref{fig:threebefore} The initial graph $G$. Edges numbered $a_1, a_2, a_3$ are dashed and coloured red.
		\protect\subref{fig:threeduring} Visualization of the package edge sliding.
		\protect\subref{fig:threeafter} The resulting graph with two updated cycles: 
		${\es_3(c)= \left[ (2 \; 3 \; 4 \; 5 \; 6) \razy \left(3 \; 4\right) \razy \left(3 \; 5\right) \razy \left(3 \; 6\right) \right] \setminus \left[ \left(4\right) \left(5\right) \left(6\right) \right] 
		= \left(2 \; 3 \right)}$ and 
		${\ess_3(c')= (3 \; 6) \razy (3 \; 5) \razy (3 \; 4) \razy \left[ \left( 1 \; 3\right) \; \left( 4 \right)\right] = (1 \; 4 \; 5 \; 6 \; 3).}$
	}
	\label{fig:sliding-three-edge}
\end{figure}
Analogously, we define the package edge sliding for the numbers $a_1, \dots, a_l$ in \emph{the counterclockwise boundary direction} from the cycle $c$ to the cycle $c'$ along \emph{the number $h=c(a)$}
as the replacement of the cycles $c, c'$ in the multiset $M_G$ by two new cycles $s_l(c), s_l(c')$ obtained by recursion
with initial conditions $\es_0(c)=c, \ess_0(c')=c'$ and a recursive step given for each $0 < j \leq l$ by
\begin{align*}
	\es_{j}(c)&= \left[ \left(a_j \; h\right) \razy \es_{j-1}(c) \right] \setminus \left(a_j\right),\\ 
	\ess_{j}(c')&= \left[ \left( a_j \right) \; \ess_{j-1} (c') \right] \razy (a_j \; h).
\end{align*}

In other words, if we would like to slide a number $a$ in the clockwise boundary direction from a cycle $c$ along a number $h$, then $c^{-1}(a)=h$ or before that, a number $c^{-1}(a)$ must be slid in the same direction. (Similarly for the counterclockwise boundary direction and number c(a).)
For each number $a$ in a cycle $c$ with a fixed direction, the number $h$ along which it will be slid is uniquely determined. Therefore, for simplicity, we will say that each of the numbers $a_1, \dots, a_l$ from the cycle $c$ is slid in a fixed direction. 
Of course, we assume that at least one number in the cycle $c$ has no direction assigned. 

\subsection{Edge sliding in the general case}

We define \emph{the edge sliding} on a graph as follows. 
We start from the graph $G$ in which some ends of certain 
edges are assigned a clockwise or counterclockwise boundary direction.
We assume that: 
\begin{itemize}
	\item At each vertex at least one edge has no direction 
	assigned.
	\item  Two consecutive edges do not have conflicting directions, 
	i.e., there is no situation in which a number $a$ from a cycle $c$ 
	has assigned the counterclockwise boundary direction and the number $c(a)$ 
	has assigned the clockwise boundary direction.
\end{itemize} 
Any such set of directions can be 
decomposed into an package edge sliding system for sequences of numbers:
$$
\begin{cases} 
	a_{1, 1}, \dots, a_{1, l_{1}} \text{ in a direction } d_{1} \text{ from the cycle } c_{1} \text{ to the cycle } c'_{1} \text{ along the number } h_{1},\\
	\hfill  \vdots  \\	
	a_{t, 1}, \dots, a_{t, l_{t}} \text{ in a direction } d_{t} \text{ from the cycle } c_{t} \text{ to the cycle } c'_{t} \text{ along the number } h_{t}. 
\end{cases}
$$
Furthermore, we require that:
\begin{itemize}
	\item The numbers, along which other numbers are sliding, 
	are not themselves sliding, i.e.,
	$$\{ a_{1,1}, \dots, a_{t, l_t}\} \cap \{ h_1, \dots, h_t \} = \emptyset.$$
	\item Two ends of the same edge will not appear in the same vertex, i.e., if $a_{i_1, j_1}=a_{i_2, j_2}$, then $$\{ c_{i_1} , {c'}_{i_1} \} \cap \{ c_{i_2} , {c'}_{i_2} \} = \emptyset .$$
	(This condition can be weakened to ${c'}_{i_1} \neq {c'}_{i_2}$, but it is not necessary in this paper.)
	\item  On one side of the edge along which we slide, the numbers do not slide in opposite directions, i.e., 
	for each number $h_j$ belonging to the cycles $c \neq c'$, there is no situation in which a number $c(h_j)$
	has assigned the clockwise direction and the number ${c'}^{-1}(h_j)$ 
	has assigned the counterclockwise direction.
\end{itemize}
We will call the selection of directions that satisfy the above conditions as \emph{correct}. 

We define the edge sliding as applying sequentially the package edge sliding for all sequences in any order.
Such an action is well-defined, since permutation multiplication is associative and 
for any distinct numbers $a_1, a_2, h_1, h_2$ holds 
$$(a_1 \; h_1) \razy (a_2 \; h_2) = (a_2 \; h_2) \razy (a_1 \; h_1) .$$
An example of the edge sliding is shown in \cref{fig:sliding}.

\begin{figure}[H]
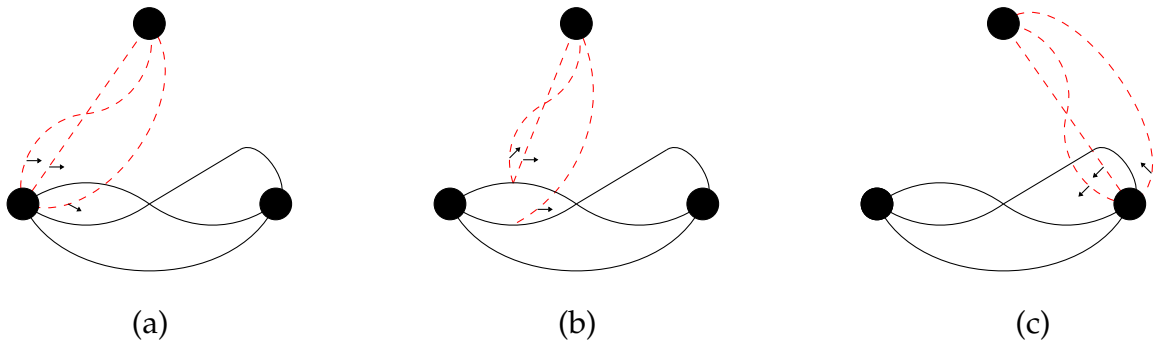

	\sliding{1.2}
	\caption
	{
		An example of edge sliding.
		\protect\subref{fig:rysprzed} 
		A graph with the slided 
		edges dashed and coloured red.
		The directions of the edge 
		ends are indicated by arrows.
		\protect\subref{fig:rysw} 
		The graph during the edge 
		sliding in the clockwise 
		boundary direction. 
		\protect\subref{fig:ryspo} 
		The graph after the edge 
		sliding. Directions have 
		already been reversed.
	}
	\label{fig:sliding}
\end{figure}
\pagebreak[3]

The edge sliding is an involution on the 
set of graphs drawn on an oriented and 
connected surface with a correct selection of directions.
Edge sliding is an invertible transformation, with the inverse 
also given by edge sliding. 

In addition, it is easy to see that the 
edge sliding on a graph does not change 
the number of faces of this graph.  

In the rest of this article, we will treat 
the edge sliding as transformation on a graph. 
\subsection{The set $X_k$ of maps} 

We consider any map from the set $X_k$. 
This map has one face 
and an odd number of edges $k \geq 5$. 
We denote the black vertices by 
$b_1, b_2$ and the white vertices by 
$w_1, w_2$. There is at least one edge
between each pair of the vertices of
different colours. Of course, 
$\degg(b_1)+\degg(b_2)=k$ is an odd 
number. Without loss of generality we 
may assume that $\degg(b_1)>0$ is an 
odd number and $\degg(b_2)>0$ is an 
even number. Let $k_1, k_2 > 0$
denote the numbers of edges which 
connect the vertex $b_1$ with the 
vertices $w_1, w_2$, respectively. 
As $\degg(b_1)=k_1+k_2$ is 
an odd number, without loss of  
generality we may assume that $k_1$ 
is even and $k_2$ is odd. For example, 
the unique (up to choice of the root) 
map from the set $X_5$ is shown in
\cref{fig:rysalla}.

\begin{figure}[H]
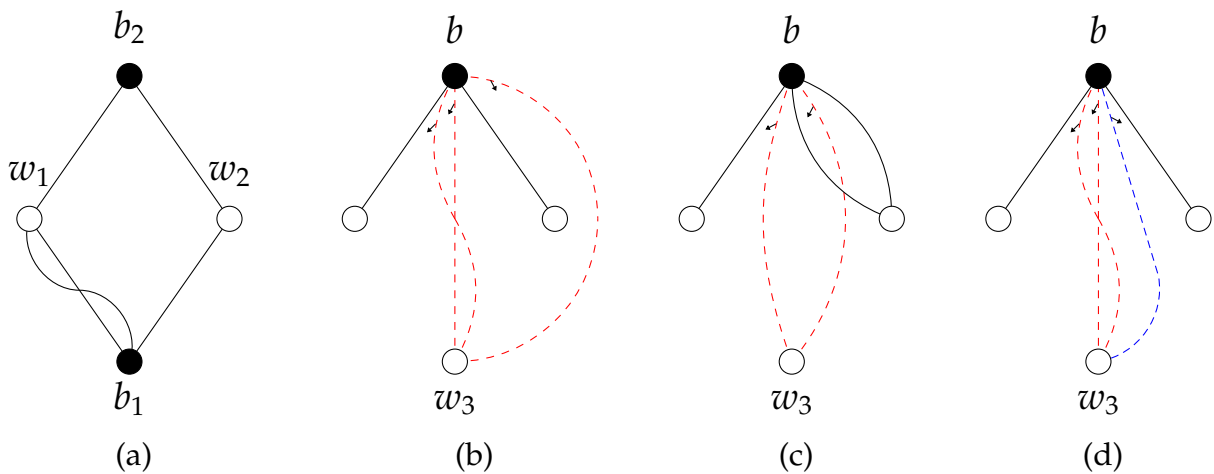

\rysall{0.95}
\caption
{
Examples of maps from the sets 
\protect\subref{fig:rysalla}  $X_5$,
\protect\subref{fig:rysallb}  $T_{5}^{\odd}$,
\protect\subref{fig:rysallc}  $T_{5}^{\even}$,
\protect\subref{fig:rysalld}  $T_{5}^{\rest}$.
Slided edges are dashed and coloured. 
The directions are indicated by arrows, and 
the root is not marked.
}
\label{fig:rysrest}
\end{figure}

\subsection{The set $Y_k$ of maps}

Let $\sigma$ be the cycle that encodes 
the clockwise boundary cyclic order of the 
corners on the unique face of $G$. 
%(this is just $\sigma_{\res}$ for $G^{\res}$ with no special edges) 
We will say 
that \emph{the vertex $w_j$ is a
descendant of the vertex $w_i$} 
(we denote it by $w_i \rightarrow w_j$) 
if using the clockwise boundary order of 
the corners on the unique face of the 
map we can move (by walking along the 
edges and holding them with the right 
hand) in two steps from a certain corner 
$c_i$ of the vertex $w_i$ to a certain
corner $c_j$ of the vertex $w_j$, i.e., 
$\nast^2(c_i)=c_j$. 

We consider any map from the set $Y_k$. 
Any such map has one face and an odd 
number of edges $k \geq 5$. We denote 
the black vertex by $b$ and  the white
vertices by $w_1, w_2, w_3$. We will 
write the set $Y_k$ as a union of three 
sets which will be defined below.

Let $Y_{k}^{\odd}\subseteq Y_k$ be the 
set of maps for which there exists an
\emph{odd degree} white vertex (let us say it 
is $w_3$) which has the other two white
vertices as descendants, i.e., 
$w_3\rightarrow w_1$ and 
$w_3\rightarrow w_2$. 
Let $T_k^{\odd}$ be the set of all 
maps from the set $Y_k^{\odd}$ with a 
distinguished vertex $w_3$ with this 
property. 
Moreover, each edge between the vertex $w_3$ and the vertex $b$, has a clockwise boundary direction assigned at the vertex $b$.
The unique (up to choice of the root) 
map from the set $T_{5}^{\odd}$ is 
shown in \cref{fig:rysallb}. Clearly
\begin{equation}
\label{ineqodd}
|T_{k}^{\odd}| \geq |Y_{k}^{\odd}|.
\end{equation}

Let $Y_{k}^{\even}\subseteq Y_k$ be 
the set of maps such that there
exists an \emph{even degree} white vertex 
(let us say it is $w_3$) which has 
the other two white vertices as 
descendants, i.e., 
$w_3\rightarrow w_1$ and 
$w_3\rightarrow w_2$. 
Let $T_k^{\even}$ be the set of all  
the maps from the set $Y_k^{\even}$ 
with a distinguished vertex $w_3$ with 
this property. 
Moreover, each edge between the vertex $w_3$ and the vertex $b$, has a clockwise boundary direction assigned at the vertex $b$.
The unique (up to choice of 
the root) map from the set $T_{5}^{\even}$ 
is shown in \cref{fig:rysallc}. Clearly
\begin{align}
\label{ineqeven}
|T_{k}^{\even}| \geq |Y_{k}^{\even}|.
\end{align}

Let $Y_{k}^{\rest}\subseteq Y_k$ be 
the set of maps not included in the
sets $Y_{k}^{\odd}$ and 
$Y_{k}^{\even}$, i.e.,
\begin{align}
\label{yrestdef}
Y_{k}^{\rest}=Y_k 
\setminus (Y_{k}^{\odd} 
\cup Y_{k}^{\even}).
\end{align}

Consider some map $m\in Y_{k}^{\rest}$. 
Obviously $w_1 \rightarrow w_2 
\rightarrow w_3 \rightarrow w_1$ or the 
other way around. Without  loss  of  
generality  we  may  assume  that 
$w_1 \rightarrow w_2 \rightarrow w_3 
\rightarrow w_1$ and as a consequence 
$w_1 \nleftarrow w_2 \nleftarrow w_3 
\nleftarrow w_1$.

\begin{lemma}
The map $m$ has a white vertex of odd 
degree, greater than $1$. 
\end{lemma}
\begin{proof}
By contradiction,  suppose this is not 
the case. The map $m$ has at least one 
odd degree white vertex, because 
$\degg(w_1)+\degg(w_2)+\degg(w_3)=k$ is 
odd. Without loss of generality we may 
assume that $\degg(w_1)$ is odd. Since 
$$\degg(w_1)+\degg(w_2)+\degg(w_3)=k>3=
1+1+1,$$ it follows that $\degg(w_1)=1$ 
and $\degg(w_2),\degg(w_3)$ are even, 
because $m$ does not have a white vertex 
with odd degree greater than $1$. The 
vertex $w_1$ is a leaf and thus has a 
unique corner which we denote by $c_1$.

Naturally $\nast^2(c_1)$ is a corner of 
the vertex $w_2$. Note that $\nast^2$ 
is a permutation of the corners of the 
white vertices which has only one cycle, 
because the map $m$ has only one face. 
The corners of the white vertices 
can be labelled 1, 2, 3 according to 
the names of the vertices 
they are in. If a corner $c$ has the 
label $a$, its descendant $\nast^2(c)$ 
has either the label $a$ or 
$1+a \operatorname{mod} 3$. 
There is only one corner which has the 
label $1$, so the corner labels 
(arranged in the cyclic order according 
to the unique cycle of $\nast^2$) are 
$(1, 2, \ldots, 2, 3, \ldots, 3)$. 
\begin{figure}[H]
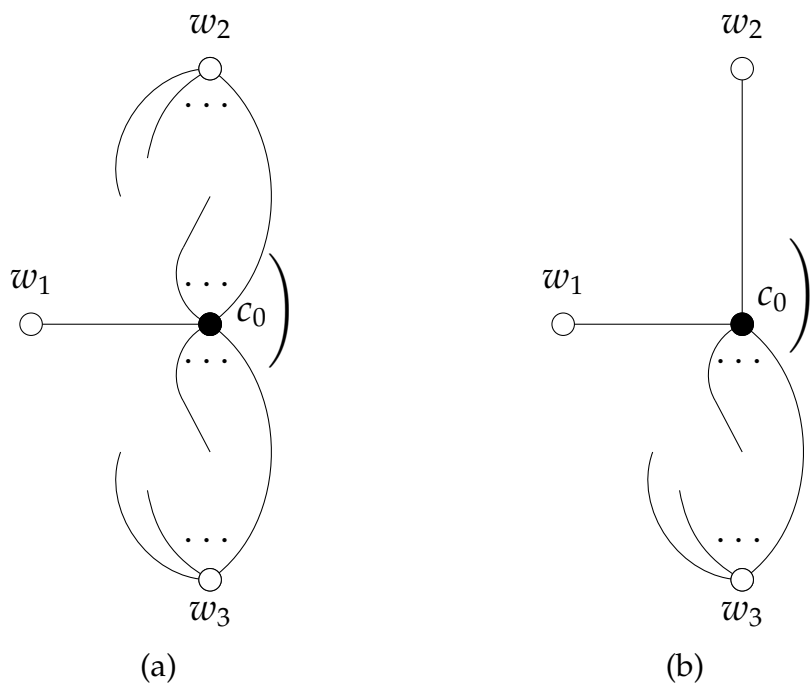

	\ryschange{0.85}
	\caption
	{ 
		\protect\subref{fig:ryschangea} The map $m$ before replacement including the marked corner $c_0$. \linebreak[4]
		\protect\subref{fig:ryschangeb} A non-existent hypothetical map with four vertices and an even number of edges. 
	}
	\label{fig:ryschange}
\end{figure}
Since 
there exists only one corner $c_2$ of the 
white vertex $ w_2 $ such that 
$\nast^2(c_2)$ is a corner of the vertex 
$w_3$, then there exists a unique corner 
$c_0=\nast(c_2)$ of the vertex $b$ such that 
$\nast(c_0)$ is the corner of the vertex $w_3$ 
and $\nast^{-1}(c_0)$ is the corner of the 
vertex $w_2$. Thus the clockwise angular cyclic 
order of the edges around the black vertex 
$b$ is as follows: one edge connected to 
the vertex $w_1$, a certain number of edges 
connected to the vertex $w_2$, a certain 
number of edges connected to the vertex $w_3$.
\cref{fig:ryschangea} visualizes this situation.

It is easy to see that we can replace all 
edges of the white vertex $w_2$ by a single 
edge (see \cref{fig:ryschange}) and get a map 
with $4$ vertices, one 
face and an even number of edges. We get a 
contradiction, because such a map does not 
exist (see \cref{euler}). Therefore, the
map $m$ has a white vertex with an odd 
degree greater than $1$.
\end{proof}

\medskip

Now, we fix the directions. 
To all ends in the vertex $b$ of the edges between the vertex $w_3$ and $b$, we assign the direction in such way that among 
them there is an even number with the clockwise boundary direction 
and an odd number with the counterclockwise boundary direction.
This is always possible, e.g. for a single edge 
with the counterclockwise boundary direction.

Let $T_k^{\rest}$ be the set of all the 
maps from the set $Y_k^{\rest}$ with a 
distinguished white vertex denoted by $w_3$ 
with a fixed choice of the set of special 
edges together with the directions of their 
ends satisfying the conditions just mentioned above. 
The unique (up to choice of the root) example 
of the map from the set $T_{5}^{\rest}$ is 
shown in \cref{fig:rysalld}. Clearly
\begin{align}
\label{ineqrest}
|T_{k}^{\rest}| \geq |Y_{k}^{\rest}|.
\end{align}

\medskip

\section{Proof of main result}

\subsection{Goal}
In this section we will construct three bijections 
which show the cardinalities of the corresponding sets 
are equal. Using these equalities and the definitions 
of these sets we will prove \cref{mainthm2}.

\subsection{Three bijections}

\paragraph{The first bijection between $X_k$ and $T_k^{\odd}$.}
We start from a map $m \in X_{k}$. Recall that we have assumed that
$\degg(b_1)$ is odd. 
All edges connecting a vertex $b_1$ to white vertices, 
have the counterclockwise boundary direction assigned at the vertices $w_1, w_2$.
The choice of directions is correct because $m$ is a bipartite graph.
We apply the 
edge sliding to the map $m$. Then we change 
the colour of the black vertex $b_1$ to white 
and its name to $w_3$, and the name of the vertex
$b_2$ to $b$. Of course, the degree of the vertex
$w_3$ does not change and is odd. In addition, 
$w_3\rightarrow w_1$ and $w_3\rightarrow w_2$,
because any map from the set $X_{k}$ has at 
least one edge between each pair of the vertices 
of different colours. We obtain a map from the 
set $T_{k}^{\odd}$. (At all times one of the edges 
is selected as the root.) Moreover, each map from the 
set $T_{k}^{\odd}$ can be produced in this way. Such a
transformation is a bijection between the set 
$X_k$ and the set $T_k^{\odd}$, 
since the edge sliding is reversible.  
\cref{fig:rysodd} shows an example of this
bijection for $k=5$. Thus
\begin{align}
\label{eqodd}
|X_k|=|T_k^{\odd}|.
\end{align}
\begin{figure}[H]
\rysodd{0.95}
\caption
{
The example of the first bijection 
for the $5$-edged map.
	\protect\subref{fig:rysodda} 
		The map from the set $X_{5}$.
	\protect\subref{fig:rysoddb} 
		The map during the edge sliding.
	\protect\subref{fig:rysoddc} 
		The map after the edge sliding.
	\protect\subref{fig:rysoddd} 
		The map from the set $Y_{5}^{\odd}$.
}
\label{fig:rysodd}
\end{figure}

\paragraph{The second bijection between $X_k$ and $T_k^{\even}$.}
We start from a map $m \in X_{k}$. Recall that we have assumed that 
$\degg(b_2)$ is even. 
All edges connecting a vertex $b_1$ to white vertices, 
have the counterclockwise boundary direction assigned at the vertices $w_1, w_2$.
The choice of directions is correct because $m$ is a bipartite graph.
We apply the 
edge sliding to the map $m$. Then we change 
the colour of the black vertex $b_2$ to white 
and its name to $w_3$, and the name of the vertex
$b_1$ to $b$. Of course, the degree of the vertex
$w_3$ does not change and is even. In addition, 
$w_3\rightarrow w_1$ and $w_3\rightarrow w_2$,
because any map from the set $X_{k}$ has at 
least one edge between each pair of the vertices 
of different colours. We obtain a map from the 
set $T_{k}^{\even}$. (At all times one of the edges 
is selected as the root.) Moreover, each map from the 
set $T_{k}^{\even}$ can be produced. Such a 
transformation is a bijection between the set 
$X_k$ and the set $T_k^{\even}$,
since the edge sliding is reversible. 
\cref{fig:ryseven} shows an example of this
bijection for $k=5$. Thus
\begin{align}
\label{eqeven}
|X_k|=|T_k^{\even}|.
\end{align}
\begin{figure}[H]
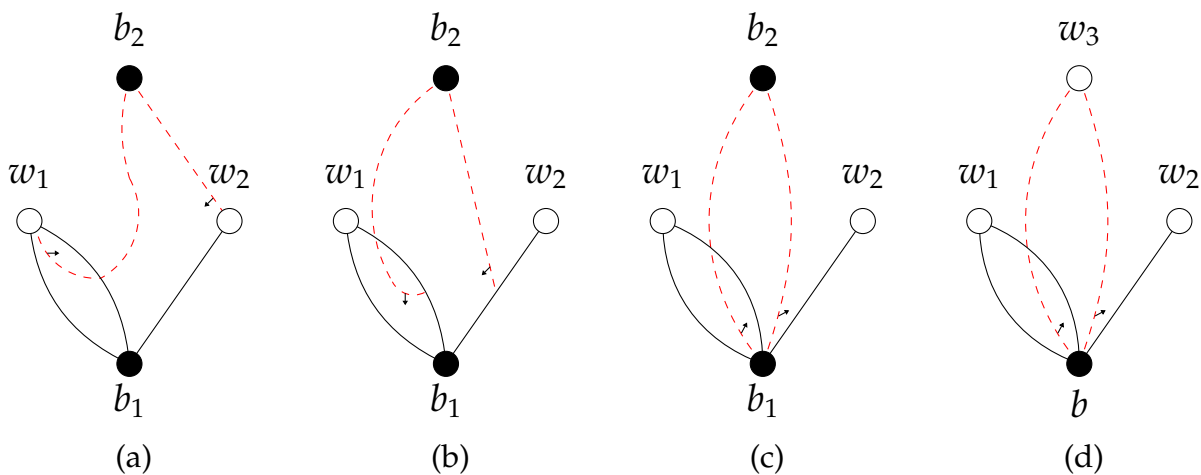

\ryseven{0.95}
\caption
{
\protect\subref{fig:rysevena}---\protect\subref{fig:rysevend} 
The example of the second bijection for the $5$-edged map. 
}
\label{fig:ryseven}
\end{figure}

\paragraph{The third bijection.}
We start from a map $m \in X_{k}$. Recall that we have assumed that
$\degg(b_1)$ is odd.
All edges connecting a vertex $b_1$ to white vertices, 
have the counterclockwise boundary direction assigned 
at the vertex $w_1$ and the clockwise boundary direction assigned at the vertex $w_2$.
The choice of directions is correct because $m$ is a bipartite graph.
We apply the 
edge sliding to the map $m$. Then we change 
the colour of the black vertex $b_1$ to white 
and its name to $w_3$, and the name of the vertex
$b_2$ to $b$. Of course, the degree of the vertex
$w_3$ does not change and is odd. In addition, 
$w_3\rightarrow w_1$ and $w_2\rightarrow w_3$ 
(and $w_1\rightarrow w_2$),
because any map from the set $X_{k}$ has at 
least one edge between each pair of the vertices 
of different colours. We do not necessarily obtain 
a map from the set $T_{k}^{\rest}$ (it may be that 
we obtain a map from set $T_{k}^{\odd}$), but it can 
be seen that each map from the set $T_{k}^{\rest}$ 
can be produced. (At all times one of the edges 
is selected as the root.)
Such a transformation is a bijection between the 
set $X_k$ and some superset of the set $T_k^{\odd}$,
since the edge sliding is reversible. 
\cref{fig:rysrest} shows an example of this
bijection for $k=5$. Thus
\begin{align}
\label{eqrest}
|X_k| \geq |T_k^{\rest}|.
\end{align}
\begin{figure}[H]
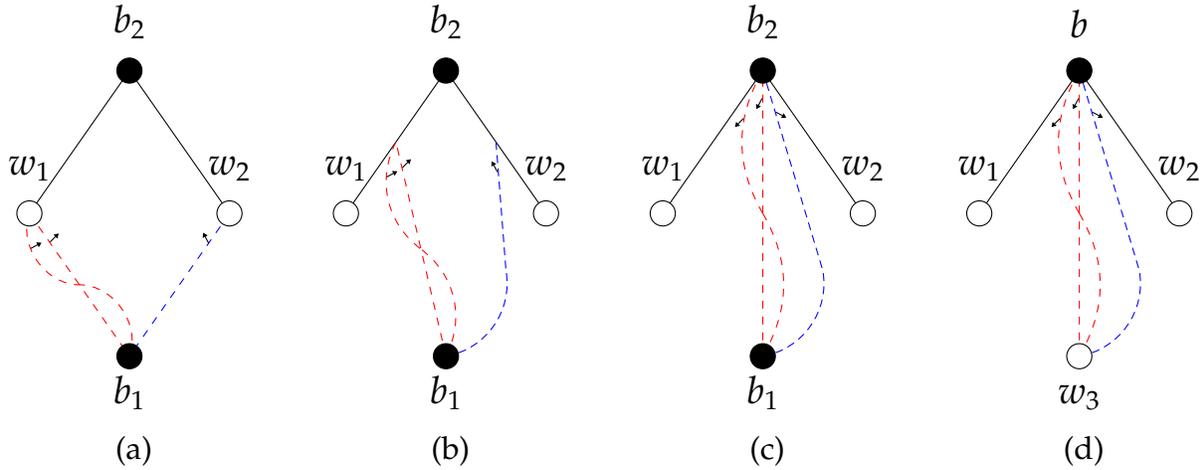

\rysrest{0.95}
\caption
{
\protect\subref{fig:rysresta} - \protect\subref{fig:rysrestd} 
The example of the third bijection for the $5$-edged map.
}
\label{fig:rysrest}
\end{figure}

\subsection{The conclusion of the proof}
\begin{proof}[\unskip\nopunct]
We can now proceed to the proof of \cref{mainthm2}, we
have: 
\begin{align*}
3[C_2^2]L_k & =3[R_2^2]K_k-[R_4]K_k \tag*{by \eqref{formula}}\\ 
& =3|X_k|-|Y_k| \tag*{by \eqref{xdef}, \eqref{ydef}}\\ 
& \geq |T_k^{\odd}|+|T_k^{\even}|+|T_k^{\rest}|-|Y_k| \tag*{by \eqref{eqodd},
\eqref{eqeven}, \eqref{eqrest}}\\ & \geq |Y_k^{\odd}|+|Y_k^{\even}|+|
Y_k^{\rest}|-|Y_k| \tag*{by \eqref{ineqodd}, \eqref{ineqeven}, \eqref{ineqrest}}\\
& =|Y_k^{\odd} \cap Y_k^{\even}| \tag*{by \eqref{yrestdef}}\\ 
& \geq 0.
\end{align*}
\end{proof}

\section*{Acknowledgements}

I thank my supervisor, Piotr
Śniady, for his support. I thank also 
Doctor Stephen Moore for the help with the 
text editing.

\printbibliography

\end{document}